\documentclass{amsart}
\usepackage{amscd,amssymb}

\newcommand{\Z}{{\text{$\mathbb Z$}}}

\newcommand{\R}{{\text{$\mathbb R$}}}
\newcommand{\C}{{\text{$\mathbb C$}}}

\newcommand{\calF}{{\text{$\mathcal F$}}} 
\newcommand{\calG}{{\text{$\mathcal G$}}}
\newcommand{\calH}{{\text{$\mathcal H$}}}
\newcommand{\cinf}{{\text{$C^\infty$}}}

\newcommand{\ar}{\rightarrow}

\newcommand{\G}{\Gamma}
\newcommand{\g}{{\text{$\mathfrak g$}}}

\newcommand{\supp}{\operatorname{supp}}
\newcommand{\codim}{\operatorname{codim}}
\newcommand{\End}{\operatorname{End}}
\newcommand{\Tr}{\operatorname{Tr}}
\newcommand{\dom}{\operatorname{dom}}

\newtheorem{thm}{Theorem}[section]
\newtheorem{lem}[thm]{Lemma}
\newtheorem{cor}[thm]{Corollary}
\newtheorem{prop}[thm]{Proposition}

\theoremstyle{definition}

\newtheorem{defn}[thm]{Definition}

\newtheorem{exmp}[thm]{Example}

\theoremstyle{remark}

\newtheorem{rem}{Remark}
                   
\newtheorem{case}{Case}
\newtheorem{ack}{Acknowledgment}   

\title[Leafwise reduced cohomology]{The dimension of the leafwise reduced cohomology} 
\author[J.A. \'Alvarez L\'opez]{Jes\'us A. \'Alvarez L\'opez} 
\address{Departamento de Xeometr\'{\i}a e Topolox\'{\i}a\\
         Facultade de Matem\'aticas\\
         Universidade de Santiago de Compostela\\
         15782 Santiago de Compostela\\ 
         Spain}
\email{jesus.alvarez@usc.es}
\thanks{Partially supported by Xunta de Galicia (Spain)}
\author[G. Hector]{Gilbert Hector}
\address{Institut Girard Desargues\\
         UPRESA 5028\\
         43, boulevard du 11 Novembre 1918\\ 
         Universit\'e Claude Bernard-Lyon~I\\  
         69622 Villeurbanne Cedex\\ 
         France}
\email{hector@geometrie.univ-lyon1.fr}


\begin{document}

\bibliographystyle{plain}

\maketitle

\begin{abstract}
Geometric conditions are given so that the
leafwise reduced cohomology is of infinite dimension, specially for foliations with dense leaves on
closed manifolds. The main new definition involved is the intersection number of subfoliations
with ``appropriate coefficients''. The leafwise reduced cohomology is also described for
homogeneous foliations with dense leaves on closed nilmanifolds.
\end{abstract}

\section{Introduction} \label{sec:introduction}
Let \calF\  be a \cinf\ foliation on a manifold $M$. The {\em leafwise de~Rham complex}
$(\Omega^{\cdot}(\calF),d_\calF)$ is the restriction to the leaves of the de~Rham complex
of $M$; i.e., $\Omega(\calF)$ is the space of differential forms on the leaves that are \cinf\ on the
ambient manifold $M$, and $d_\calF$ is the de~Rham derivative on the leaves. We use the notation 
$\Omega(\calF)=\cinf(\bigwedge T^{\ast}\calF)$ meaning \cinf\ sections on $M$. The cohomology
$H^{\cdot}(\calF )=H^{\cdot}(\Omega(\calF ),d_\calF )$ is called the {\em leafwise cohomology} of \calF.
It is well known that $H^{\cdot}(\calF)$ can also be
defined as the cohomology of $M$ with coefficients in the sheaf of germs of \cinf\ functions which are
locally constant on the leaves, but we do not use this. The (weak) $\cinf $ topology on
$\Omega(\calF )$ induces a topology on $H^{\cdot}(\calF)$, which is non-Hausdorff in general
\cite{Haefliger80}. The quotient space of $H^{\cdot}(\calF)$ over the closure of its trivial subspace is
called the {\em leafwise reduced cohomology\/} of \calF,  and denoted by $\calH^\cdot(\calF)$.
Similarly, we can also define $\Omega_c^{\cdot}(\calF )$,
$H_c^{\cdot}(\calF)$ and 
$\calH_c^{\cdot}(\calF )$ by considering compactly supported \cinf\ sections of 
$\bigwedge T\calF^{\ast}$.

For degree zero we have that $H^0(\calF)=\calH^0(\calF)$ is the space of \cinf\ functions on $M$ that
are constant on each leaf---the so called (smooth) {\em basic functions\/}; thus $H^0(\calF)\cong\R$ if
the leaves are dense. Though density of the leaves seems to yield strong restrictions on the leafwise
cohomology also for higher degree, this cohomology may be of infinite
dimension when leaves are dense and $M$ is closed. In fact, for dense linear flows on the
two-dimensional torus, we have
$\dim H^1(\calF )=1$ when the slope of the leaves is a diophantine irrational number \cite{Heitsch75},
but
$\dim H^1(\calF )=\infty$ if the slope is a Liouville's irrational number \cite{Roger}. Nevertheless
$\calH^1(\calF )\cong\Bbb R$ in both cases. This computation was later generalized to the case of 
linear foliations on tori of arbitrary dimension \cite{KacimiTihami,ArrautdosSantos}. 

Other known properties of the leafwise cohomology are the following ones. The leafwise cohomology
of degree one with coefficients in the normal bundle is related to the infinitesimal
deformations of the foliation \cite{Heitsch75}. For $p=\dim \calF$, the dual 
space $H_c^p(\calF)'$ is canonically isomorphic to the space of holonomy invariant transverse
distributions \cite{Haefliger80}---recall that for a topological vector space $V$, the dual space $V'$ is
the space of continuous linear maps $V\ar\R$.
$H^\cdot(\calF )$ is invariant by leaf preserving homotopies, and Mayer-Vietoris arguments can be
applied \cite{Kacimi83}, which was used to compute $H^\cdot(\calF )$ for some examples. For an arbitrary
flow \calF\  on the two-torus, it was proved that $\dim H^1(\calF )=\infty$ if \calF\  is not minimal,
and $\dim H^1(\calF )=1$ if and only if \calF\  is $\cinf $ conjugate to a Diophantine linear flow
\cite{AndradePereira}. The triviality of $H^1(\calF )$ implies the triviality of the linear holonomy
\cite{Kacimi83}, and is equivalent to Thurston's stability if $\codim\calF=1$ and $M$ is closed
\cite{AndradeHector}. However, more general relations between the leafwise cohomology and the geometry
of the foliation remain rather unknown.

The above examples of linear foliations on tori could wrongly suggest that $\calH^\cdot(\calF)$ may be of
finite dimension if $M$ is closed and the leaves are dense. In fact, S.~Hurder and the
first author gave examples of foliations with dense leaves on closed Riemannian manifolds 
with an infinite dimensional space of leafwise harmonic forms that are \cinf\ on the
ambient manifold \cite{AlvHurder}, and this space is canonically injected in the leafwise reduced
cohomology; indeed this injection is an isomorphism at least for the so called Riemannian foliations
\cite{AlvKordy1}. So a natural problem is the following: {\em Give geometric properties characterizing
\cinf\ foliations whose leafwise reduced cohomology is of finite dimension; specially for foliations
with dense leaves on closed manifolds\/}.

The aim of this paper is to give an approach to this problem. The first and main geometric idea we
use is the {\em intersection number\/} of subfoliations with ``appropriate coefficients''. To explain
it, consider the simplest example where $M=T\times L$ with the foliation \calF\ whose leaves are
the slices $\{\ast\}\times L$, where $T,L$ are closed manifolds of dimensions $q,p$. Let
$(\Omega^\cdot(L),d_L)$ be the de~Rham complex of $L$, and let $H_\cdot(L),H^\cdot(L)$ denote the
homology and cohomology of $L$ with real coefficients. Then $\Omega^\cdot(\calF)$ is the
\cinf~closure of $\cinf(T)\otimes\Omega^\cdot(L)$, where $d_\calF=1\otimes d_L$. So
$$
\calH^k(\calF)=H^k(\calF)\equiv\cinf(T)\otimes H^k(L)
$$
because $H^\cdot(L)$ is of finite dimension. Assume $L$ is oriented for
simplicity.  Then recall that Poincar\'e duality and integration of differential forms establish canonical
isomorphisms 
$$H^k(L)\cong H_{q-k}(L)\;,\quad H^k(L)'\cong H_k(L)\;,$$
such that the canonical pairing between $H^k(L)\otimes H^k(L)'\ar\R$ corresponds to the intersection
pairing $H_{p-k}(L)\otimes H_k(L)\ar\R$ \cite{Thom}. Hence
\begin{align}
\calH^k(\calF)&\cong\cinf(T)\otimes H_{p-k}(L)\;,\label{e:Hk(F)}\\
\calH^k(\calF)'&\cong\cinf(T)'\otimes H_k(L)\;,\label{e:Hk(F)'}
\end{align}
such that the canonical pairing $\calH^k(\calF)\otimes\calH^k(\calF)'\ar\R$ corresponds to the product 
of the evaluation of distributions on \cinf\ functions and the intersection pairing. Now observe that,
according to \cite{Thom}, the right hand side spaces in \eqref{e:Hk(F)} and \eqref{e:Hk(F)'} are
respectively generated by elements of the form
$f\otimes [K_1]$ and $D\otimes [K_2]$, where $f$ is a
\cinf\ function on $T$, $D$ is a distribution on $T$, and $K_1,K_2\subset L$ are closed oriented
submanifolds of dimensions $p-k,k$. Hence $\dim\calH^k(\calF)=\infty$ is equivalent to the existence of
sequences of elements $f_m\otimes [K_1]$ and $D_n\otimes [K_2]$ as above so that $K_1,K_2$
have non-trivial intersection number and $D_n(f_n)\neq0$ if and only if $m=n$;
of course this holds just when $H^k(L)\neq0$ and $q>0$. 

Now consider each element $f_m\otimes [K_1]$ as the family of
homology classes 
$$f_m(t)\,[\{t\}\times K_1]\in H_{p-k}(\{t\}\times L)\;,\quad t\in T\;,$$ 
determined by the 
family of closed oriented submanifolds $\{t\}\times K_1$ of the leaves of \calF\ and the family of
coefficients $f_m(t)$. The elements $D_n\otimes [K_2]$ have a similar interpretation by considering
distributions as generalized functions.  
A key property here is that the families $\{t\}\times K_1$ and $\{t\}\times K_2$ depend
smoothly on $t$, determining \cinf\ subfoliations $\calF_1,\calF_2$ of \calF. Other key properties are
the \cinf\ dependence of the coefficients $f_m(t)$ on $t$, and the distributional dependence of the
generalized coefficients $D_n(t)$ on $t$. This means that the $f_m$ are
\cinf\ basic functions of $\calF_1$ and the $D_n$ are ``distributional basic functions'' of $\calF_2$;
i.e., the $D_n$ are holonomy invariant transverse distributions of $\calF_2$. 
It turns out that these key properties are enough to generalize the above ideas in a way that can be
applied even when the leaves are dense, obtaining our first main theorem that roughly asserts the
following: {\em For a
\cinf\ oriented foliation \calF\ of dimension $p$, we have $\dim\calH_c^k(\calF)=\infty$ when \calF\ has
oriented subfoliations $\calF_1,\calF_2$ of dimensions $k-p,p$, and there is a sequence
of basic functions $f_m$ of $\calF_1$ and a sequence of transverse invariant
distributions $D_n$ of $\calF_2$, such that the corresponding
``intersection numbers'' are non-trivial if and only if $m=n$---certain simple conditions are also
required for the ``intersection numbers'' to be defined\/}. We do not know whether such conditions form a
characterization of the cases where
$\dim\calH_c^k(\calF )=\infty$; this depends on whether it is possible to ``smoothen'' the
representatives of classes in certain leafwise homologies introduced in
\cite{AlvHector2}. Indeed the above $f_m$ and
$D_n$ play the r\^ole of coefficients in homology, assigning a number to each leaf of the
subfoliations; the way these numbers vary from leaf to leaf is what makes these
coefficients appropriate.

Though these conditions are difficult to check in general, this result has many
easy to apply corollaries. For instance, suppose an oriented foliation \calF\ is {\em
Riemannian\/}---in the sense that all of its holonomy transformations are local isometries for some
Riemannian metric on local transversals \cite{Reinhart59,Molino88}. Then
$\dim\calH_c^\cdot(\calF)=\infty$ if
\calF\ is of positive codimension and some leaf of \calF\ contains homology classes with non-trivial
intersection. These conditions are quite simple to verify. In this case, the infinitely many linearly
independent classes obtained in $\calH_c^\cdot(\calF)$ can be considered as ``transverse diffusions'' of
the homology classes in the leaf. This diffusion idea is inspired by the unpublished preprints
\cite{Hurder:spectral,AlvHurder}. Indeed \cite{Hurder:spectral} is the
germinal work about the relation of the analysis on the leaves and on the ambient manifold
obtained by transverse diffusion.

Other consequences of the above general theorem hold when \calF\ is a {\em suspension foliation\/}. That
is, the ambient manifold of \calF\ is the total space of a fiber bundle $M\ar B$ with the leaves 
transverse to the fibers, and such that the restriction of the bundle projection to each leaf is a regular
covering of the base $B$. Now $\dim\calH_c^\cdot(\calF)=\infty$ when $B$ is oriented and has
homology classes with non-trivial intersection satisfying additional properties with respect to
the holonomy of \calF. In this case the leaves may not contain homology classes with
non-trivial intersection, and thus the idea of ``transverse diffusion'' of homology classes in the
leaves may fail. In fact we shall see that the infinite dimension of $\calH_c^1(\calF)$
may be more related to the number of ends of the leaves.

To explain another theorem of this paper, recall that a foliation \calF\ on a manifold $M$ is a {\em Lie
foliation\/} when it has a complete transversal diffeomorphic to an open subset of a Lie group
$G$ so that holonomy transformations on this transversal correspond to restrictions of left
translations on $G$---this type of foliations play a central r\^ole in the study of Riemannian
foliations \cite{Molino88}. The Lie algebra \g\ of $G$ is called the {\em structural Lie algebra} of
\calF; we may also simply say that \calF\ is a Lie \g-foliation.
In this case, if $M$ is closed and oriented, and \g\ is compact semisimple, then we obtain that
$\dim\calH^\cdot(\calF)=\infty$ when some additional hypotheses are satisfied. Again we use homology
classes with non-trivial intersection in the hypotheses, but now they live in the homology of $M$. The
proof of this result is reduced to the case of suspension foliations to apply what we already know. This
reduction process contains rather delicate arguments based on the work \cite{Alv5} of the first author. 

The above results are negative in the sense that all of them give conditions for the nonexistence of
finite Betti numbers for the reduced leafwise cohomology. In contrast, our final theorem shows that the
reduced leafwise cohomology of the so called {\em homogeneous foliations\/} with dense leaves in closed
nilmanifolds is isomorphic to the cohomology of the Lie algebra defining the foliation.  This has been
also proved by X.~Masa with different techniques.

\begin{ack}
We wish to thank F.~Alcalde for many helpful conversations. The first author would
like to thank the hospitality of the Institut de  Ma\-th\'e\-ma\-ti\-ques et
d'Informatique of the University Claude Bernard of Lyon several times during the preparation of
this work. We would like also to thank the referee for important corrections.
\end{ack}

\section{Main results}\label{sec:main results}

For the sake of simplicity, all manifolds, foliations, maps, functions, differential
forms and actions will be assumed to be \cinf\ from now on, unless the contrary is
explicitly stated. 

Let \calF\ be a foliation on a manifold $M$. For any closed saturated subset $S\subset M$, let 
$\Omega_S^{\cdot}(\calF )\subset\Omega^{\cdot}(\calF )$ be the subcomplex of leafwise
differential forms whose support has compact intersection with $S$.
Consider the topology on $\Omega_S^{\cdot}(\calF )$ determined as follows: A 
sequence $\alpha_n\in\Omega_S^{\cdot}(\calF )$ converges to zero if it converges to zero in
$\Omega^\cdot(\calF)$ and there is a compact subset $K\subset S$ such that $S\cap\supp\alpha_n\subset
K$ for all $n$. We have the corresponding cohomology $H_S^{\cdot}(\calF)$, and reduced cohomology 
$\calH_S^{\cdot}(\calF)$. With this notations, observe that
$\Omega^{\cdot}(\calF)=\Omega_\emptyset^{\cdot}(\calF)$ and
$\Omega_c^{\cdot}(\calF)=\Omega_M^{\cdot}(\calF)$ as topological vector spaces.

Let $f:(M_1,\calF_1)\ar (M_2,\calF_2)$ be a map of foliated manifolds, and  let
$S_i\subset M_i$, $i=1,2$, be closed saturated subsets such that the restriction $f:S_1\ar S_2$ is a
proper map. Then $f^\ast(\Omega_{S_2}^\cdot(\calF_2))\subset\Omega_{S_1}^\cdot(\calF_1)$,
yielding a homomorphism $f^\ast:\calH_{S_2}^{\cdot}(\calF_2)\ar\calH_{S_1}^{\cdot}(\calF_1)$. In
particular we get $f^\ast:{\calH}_c^{\cdot}(\calF_2)\rightarrow\calH_{S_1}^{\cdot}(\calF_1)$ if
$f:S_1\rightarrow M_2$ is proper.

The following is what we need to define the intersection number of subfoliations with
``appropriate coefficients'':   
\begin{itemize}

\item An oriented foliation \calF\  on a manifold $M$, and two immersed oriented subfoliations
$\iota_i:(M_i,\calF_i)\rightarrow(M,\calF )$, $i=1,2$.

\item $\dim \calF =\dim \calF_1+\dim \calF_2$, and $\codim\calF = \codim \calF_1$.

\item Each $\iota_i$  is {\em transversely regular} in
the sense that it defines embeddings of small enough local transversals of $\calF_i$ into local
transversals of \calF; i.e. the homomorphism defined by the differential of $\iota_i$ between
the normal bundles of $\calF_i$ and \calF\  is injective on the fibers.
  
\item  A compactly supported basic function $f$ of $\calF_1$.

\item A holonomy invariant transverse distribution $D$ of $\calF_2$ such that the map $\iota_2:
\supp D\ar M$ is proper.

\item The restrictions $\iota_1|_{\supp f}$ and $\iota_2|_{\supp D}$ {\em
intersect transversely} in \calF\  in the sense that, for all leaves $L_i$ of
$\calF_i$ and $L$ of \calF\  such that $L_1\subset\supp f$, $L_2\subset\supp D$
and  $\iota_1(L_1)\cup\iota_2(L_2)\subset L$, the immersed submanifolds
$\iota_i:L_i\ar L$ intersect transversely in $L$. 

\end{itemize}  
Observe that there are open neighborhoods, $N_1$ of
$\supp f$ and $N_2$ of $\supp D$, such that the $\iota_i|_{N_i}$ intersect
transversely in \calF. Consider the pull-back diagram  
$$\begin{CD}
T                @>{\sigma_1}>>      N_1\\
@V{\sigma_2}VV                       @VV{\iota_1}V\\
N_2              @>{\iota_2}>>       M\;.
\end{CD}$$
Here
$$T=\{(x_1,x_2)\in N_1\times N_2\ |\ \iota_1(x_1)= \iota_2(x_2)\}\;,$$
and the $\sigma_i$ are restrictions of the factor projections. It is easy to
check that $\iota_1\times\iota_2: N_1\times N_2 \ar M\times M$ is transverse to the
diagonal $\Delta$, and thus $T$ is a manifold with $\dim
T=\codim \calF_2$. Moreover the $\sigma_i$ are immersions, and $\sigma_2$ is
transverse to $\calF_2$. So $D$ defines a distribution on $T$, which will be denoted by
$D_T$. We also have the locally constant {\em intersection function}
$\varepsilon:T\rightarrow\{\pm 1\}$, where 
$\varepsilon(x_1,x_2)=\pm 1$ depending on whether the identity
$$ T_{\iota_i(x_i)}\calF\equiv\iota_{1\ast}
 T_{x_1}\calF_1\oplus\iota_{2\ast} T_{x_2}\calF_2$$ is orientation
preserving or orientation reversing. On the other hand   
$$(\iota_1(\supp f)\times\iota_2(\supp D))\cap\Delta$$
is compact because it is a closed subset of the compact space $\iota_1(\supp f)\times\iota_1(\supp f)$.
So 
$$\supp \sigma_1^\ast f\cap\supp D_T=
(\iota_1\times\iota_2)^{-1}((\iota_1(\supp f)\times\iota_2(\supp D))\cap\Delta)$$
is a compact subspace of $T$ since $\iota_1\times\iota_2:\supp f\times\supp D\ar M\times M$ is a proper
map. Thus the following definition makes sense.

\begin{defn} \label{defn:intersection}
With the above notations, the {\em intersection number} of $(\iota_1,f)$ and $(\iota_2,D)$, 
denoted by $\langle(\iota_1,f),(\iota_2,D)\rangle$, is defined as
$D_T(g)$ for any compactly supported function $g$ on $T$ which is equal to the product
$\varepsilon\,\sigma_1^\ast f$
on some neighborhood of $\supp \sigma_1^\ast f\cap\supp D_T$. \end{defn}

Now our first main theorem is the following.

\begin{thm}\label{thm:subfoliations}
Let \calF\  be an oriented foliation on a manifold $M$, and $\iota_i:(M_i,\calF_i)\ar (M,\calF )$,
$i=1,2$, transversely regular immersed oriented subfoliations. Suppose $\dim\calF
=\dim\calF_1+\dim\calF_2$, and $\codim \calF =  \codim \calF_1$. Let $f_m$ be
a sequence of compactly supported basic functions of $\calF_1$, and $D_n$ a sequence of
holonomy invariant transverse distributions of $\calF_2$ such that each restriction
$\iota_2: \supp D_n\ar M$ is a proper map. Suppose each pair $\iota_1|_{\supp f_m}$ and
$\iota_2|_{\supp D_n}$ intersect transversely in \calF,  and 
$\langle(\iota_1,f_m),(\iota_2,D_n)\rangle\neq 0$ if and only if $m=n$.  Then $\dim\calH_c^k(\calF
)=\infty$ for $k=\dim\calF_2$. 
\end{thm}

The following two corollaries are the first type of consequences of Theorem~\ref{thm:subfoliations}; the
second corollary follows directly from the first one. 

\begin{cor}\label{cor:homology1}
Let \calF\  be an oriented foliation of codimension $q>0$, $L$ a leaf of \calF,  and 
$h:\pi_1(L)\ar G_q^\infty$ its holonomy representation,
where $G_q^\infty$ is the group of germs at the origin of local 
diffeomorphisms of $\R^q$ with the origin as fixed point. Let
$\iota_i:K_i\ar L$, $i=1,2$, be smooth immersions of closed  oriented
manifolds of complementary dimension and nontrivial intersection. Suppose there is a
Riemannian metric on
$\R^q$ so that the elements in the image of the composites 
\begin{equation}\label{e:composite}
\begin{CD}
\pi_1(K_i)@>\pi_1(\iota_i)>>\pi_1(L) @>h>>G_q^\infty
\end{CD}\end{equation}
are germs of local isometries. Then
$\dim\calH_c^{k_i}(\calF )=\infty$ for $k_i=\dim K_i$, $i=1,2$. 
\end{cor}

\begin{cor}\label{cor:homology2}
Let \calF\  be an oriented Riemannian foliation of positive codimension. Suppose some leaf of \calF\ 
has  homology classes of complementary degrees, $k_1$ and $k_2$, with non-trivial intersection.
Then 
$\dim\calH_c^{k_i}(\calF )=\infty$, $i=1,2$.
\end{cor}

Before stating the next type of corollaries of Theorem~\ref{thm:subfoliations}, recall that a suspension
foliation \calF\ is given as follows. Let
$\pi:L\ar B$ be a regular covering map of an
oriented manifold, and let $\G$ be its group of deck transformations. For any effective action of
$\Gamma$  on some manifold $T$, consider the right diagonal action of $\Gamma$ on $L\times
T$: $(z,t)\gamma = (z\gamma,\gamma^{-1}t)$ for $\gamma \in \G$ 
and $(z,t) \in L\times T$. Then \calF\  is the foliation on $M = (L\times
T)/\G$  whose leaves are the projections of the submanifolds $L\times
\{\ast\} \subset L\times T$. The element in $M$ defined by each $(z,t) \in
L\times T$ will be denoted by $[z,t]$. The map 
$\rho:M \ar B$ given by $\rho([z,t]) = \pi(z)$ is a
fiber bundle projection with typical fiber $T$. The leaves of \calF\  are transverse to
the fibers of $\rho$, and define coverings of $B$. The leaf that contains $[z,t]$ can be
canonically identified to $L/\G_t$, where $\G_t$ is the isotropy
subgroup of $\G$ at $t$. This leaf is dense if and only if the $\G$-orbit of $t$ is dense
in $T$.

\begin{cor}\label{cor:suspension1}
With the above notation, let $h:\pi_1(B)\rightarrow\Gamma$ be the surjective homomorphism defined by the
regular covering $L$ of $B$, and let $\iota_i:K_i\ar B$, $i=1,2$, be immersions of
connected oriented manifolds of complementary dimension in $B$. Suppose $K_1$ is a closed
manifold, $\iota_2$ a proper map, and the homology class defined by $\iota_1$ has non-trivial
intersection with the locally finite homology class defined by $\iota_2$. For each $i$, let
$\Gamma_i\subset\Gamma$ be the image of the composite
$$\begin{CD}
\pi_1(K_i)@>\pi_1(\iota_i)>>\pi_1(B)
@>h>>\Gamma\;.
\end{CD}$$
Let $f_m$ be a sequence of compactly supported $\Gamma_1$-invariant functions on $T$, and $D_n$
a sequence of $\Gamma_2$-invariant distributions on $T$ such that $D_n(f_m)\neq 0$ if and only
if $m=n$. Then $\dim\calH_c^k(\calF )=\infty$ for $k=\dim K_2$.
\end{cor}

\begin{cor}\label{cor:suspension2}
Let $B$, $L$, $h$, $\Gamma$, $T$, \calF,  $K_i$, $\iota_i$ and $\Gamma_i$ be as in
Corollary~\ref{cor:suspension1}. Let $\mu$ be a $\Gamma_2$-invariant
measure on $T$. Suppose the closure of the image of $\Gamma_1$ in the topological group of diffeomorphisms
of $T$ {\rm (}with the weak \cinf\ topology{\rm )} is a compact Lie group, and there is an infinite
sequence of
$\Gamma_1$-invariant open subsets of
$T$ with non-trivial $\mu$-measure and pairwise disjoint $\Gamma_2$-saturations. Then
$\dim{\calH}_c^k(\calF )=\infty$ for $k=\dim K_2$.   
\end{cor}

Observe that,  in Corollary~\ref{cor:suspension2}, the infinite sequence of $\G_1$-invariant open
sets may not be $\G_2$-invariant, and their $\Gamma_2$-saturations may not be $\Gamma_1$-invariant.

\begin{cor}\label{cor:suspension3}
Let $B$, $L$, $h$, $\Gamma$, $T$ and \calF\  be as in Corollary
\ref{cor:suspension1}. Suppose that there is a loop $c:S^1\to B$
with a lift to $L$ that joins two distinct points of the end set of $L$. Let $a=h([c])\in\G$, where
$[c]$ is the element of $\pi_1(B)$ represented by $c$, and assume that the closure $H$ of the image of
$\langle a\rangle$ in the topological group of diffeomorphisms
of $T$ {\rm (}with the weak \cinf\ topology{\rm )} is a compact Lie group. Suppose also that there
is an infinite sequence of disjoint non-trivial $H$-invariant open subsets of $T$. Then
$\dim\calH_c^1(\calF )=\infty$.
\end{cor}

In Corollaries~\ref{cor:suspension1},~\ref{cor:suspension2} and~\ref{cor:suspension3}, 
if $B$ is compact, then the leaves of \calF\ 
can only be dense when $L$ has either one end or a Cantor space of ends, as follows from the
following.

\begin{prop}\label{prop:orbit density} 
Let $\Gamma$ be a finitely generated group with two ends, and $C\subset \Gamma$ an
infinite subgroup. Suppose $\Gamma$ acts continuously on some connected $T_1$
topological space $X$. Then the $\Gamma$-orbits are dense in $X$ if and only if so are the
$C$-orbits.  \end{prop}

Now let \calF\ be a Lie \g-foliation on a closed manifold $M$. The 
following property characterizes such a type of foliations \cite{Fedida,Molino82,Molino88}. Let
$\widetilde M$ be the universal covering of $M$, $\widetilde\calF$ the lift of \calF\ to
$\widetilde M$, and $G$ the simply connected Lie group with Lie algebra \g. Then the leaves of
$\widetilde\calF$ are the fibers of a fiber bundle $\widetilde M\ar G$, which is
equivariant with respect to some homomorphism $h:\pi_1(M)\ar G$, where we consider the right action of
$\pi_1(M)$ on $\widetilde M$ by deck transformations and the right action of $G$ on itself by right
translations. This $h$ and its image are respectively called the {\em holonomy homomorphism\/} and {\em
holonomy group\/} of \calF. Observe that the fibers of $D$ are connected because $G$ is simply connected
(a connected covering of $G$ is given by the quotient of $\widetilde M$ whose points are the connected
components of these fibers).

\begin{thm}\label{thm:css}
With the above notation, suppose that $M$ is oriented and
the structural Lie algebra $\g$ of \calF\  is compact semisimple. Let $\iota_i:K_i\ar M$, $i=1,2$, be
immersions of closed oriented manifolds of complementary dimension defining homology
classes of $M$ with non-trivial intersection. Let $\Gamma_i$ be the image of the
composite 
$$\begin{CD}
\pi_1(K_i)@>\pi_1(\iota_i)>>\pi_1(M)@>h>>G\;.
\end{CD}$$ 
Suppose the group generated by $\Gamma_1\cup\Gamma_2$ is not dense in $G$. Let
$k=\dim K_2$, and suppose either $1\leq k\leq 2$ or $\iota_1$ is transverse to \calF.
Then $\dim\calH^k(\calF )=\infty$. 
\end{thm}

The following is our final theorem.

\begin{thm}\label{thm:nilpotent}
Let $H$ be a simply connected nilpotent Lie group, $K\subset H$ a normal connected subgroup, and
$\Gamma\subset H$ a discrete uniform subgroup whose projection to $H/K$ is dense. Let \calF\  be
the foliation of the closed nilmanifold $\Gamma\backslash H$ defined as the quotient of the
foliation on $H$ whose leaves are the translates of $K$. Then there is a canonical isomorphism
$\calH^{\cdot}(\calF )\cong H^{\cdot}(\frak{k})$, 
where $\frak{k}$ is the Lie algebra of $K$. 
\end{thm}

The following two examples are of different nature. In both of them there are 
infinitely many linearly independent leafwise reduced cohomology classes of degree one. 
But these classes are induced by the handles in the leaves in Example~\ref{exmp:AlvHurder}, 
whereas they
are induced by the ``branches'' of the leaves that define a Cantor space of ends in
Example~\ref{exmp:Cantor}.

\begin{exmp}[{\cite{AlvHurder}}]\label{exmp:AlvHurder}
Let $L$ be a $\Bbb Z$-covering of the compact oriented surface of genus two; i.e., 
$L$ is a
cylinder with infinitely many handles attached to it. Each handle contains two circles defining
homology classes with non-trivial intersection. Hence for any injection of $\Bbb Z$ into the 
$n$-torus
$\R^n/\Z ^n$, the corresponding suspension foliation fulfills the hypotheses of
Corollary~\ref{cor:suspension2}, and thus has infinite dimensional reduced
leafwise cohomology of degree one. We could also use
Corollary~\ref{cor:homology2} instead of Corollary~\ref{cor:suspension2}.
\end{exmp} 

\begin{exmp}\label{exmp:Cantor} Let $\G$ be the free group with two generators, and $L$ a
$\G$-covering of the compact orientable surface of genus two. This $L$ has a Cantor space
of ends. Hence, for any injective homomorphism of $\G$ in a compact Lie group, the
reduced leafwise cohomology of degree one of the corresponding suspension foliation is of
infinite dimension by Corollary~\ref{cor:suspension3}.
\end{exmp}

\section{Leafwise reduced cohomology and subfoliations}\label{sec:subfoliations}

This section is devoted to the proof of Theorem~\ref{thm:subfoliations}. With the notations introduced in
Section~\ref{sec:main results}, the idea of the proof is the following. The
$(\iota_1,f_m)$ yield elements $\zeta_m\in\calH_c^r(\calF)$ by ``leafwise Poincar\'e duality''. On the
other hand, the arguments in \cite{Haefliger80} show that each $D_n$ can be considered as an element in
$\calH_{S_n}^r(\calF_2)'$, where $S_n=\supp D_n\subset M$. Moreover there are  homomorphisms
$\iota_2^\ast:\calH_c^\cdot(\calF )\ar\calH_{S_n}^\cdot(\calF_2)$ since the $\iota_2: S_n\ar M$ are
proper maps. Then the result follows by verifying
$\langle(\iota_1,f_m),(\iota_2,D_n)\rangle=D_n(\iota_2^\ast\zeta_m)$. 

We first explain the way ``leafwise Poincar\'e duality'' works. 
Consider the {\em transverse complex} $\Omega_c^\cdot(\Tr\calF)$ introduced in \cite{Haefliger80}, which
will be only used for degree zero. 
For any representative \calH\ of the holonomy pseudogroup of \calF, defined on some manifold $T$, 
$\Omega_c^0(\Tr\calF)$ is defined as the quotient of
$\cinf_c(T)$ over the subspace generated by the functions of the type
$\phi-h^\ast\phi$, where 
$h\in\calH$ and $\phi\in\cinf_c(T)$ with $\supp\phi\subset\dom h$.
 As a topological vector space, 
$\Omega_c^0(\Tr\calF)$ is independent of chosen representative of the holonomy pseudogroup. 
From the definition it easily follows that the dual space $\Omega_c^0(\Tr\calF )'$ can be
canonically identified to the space of holonomy invariant transverse
distributions of \calF. 
  
Now consider the representative \calH\ of the holonomy pseudogroup induced by an appropriate locally
finite covering of
$M$ by foliation patches $U_i$; that is, if $f_i:U_i\ar T_i$ is the local quotient map whose fibers
are the plaques in $U_i$, then appropriateness of this covering means that each equality
$f_j=h_{i,j}f_i$ on $U_i\cap U_j$ determines a diffeomorphisms $h_{i,j}:f_i(U_i\cap U_j)\ar f_j(U_i\cap
U_j)$, and the collection of all of these diffeomorphisms generate the pseudogroup 
\calH\ on $T=\bigsqcup_iT_i$. Fix also a partition
of unity
$\phi_i$ subordinated to the covering $U_i$. With these data we have a map
$\Omega_c^p(\calF)\ar\Omega_c^\cdot(T)$ given by
$\alpha\mapsto\sum_i\int_{f_i}\phi_i\alpha$, where $p=\dim\calF$ and $\int_{f_i}$ denotes integration 
along the fibers of $f_i$. This ``integration along the leaves'' induces an isomorphism 
$H^p(\calF)\cong\Omega_c^0(\Tr\calF)$ of topological
vector spaces, which is independent of the choice of the $U_i$ and $\phi_i$ \cite[\S 3.3]{Haefliger80}.
So 
$$\calH_c^p(\calF)'\equiv H_c^p(\calF )'\cong \Omega_c^0(\Tr\calF )'\;;$$
i.e., any holonomy invariant distribution $D$ can be
canonically considered as an element in $\calH_c^p(\calF)'$. 
Moreover $D$ can be also considered as an element in
$H_S^p(\calF )'\equiv\calH_S^p(\calF)'$ for $S=\supp D$ as follows from the following argument.
For any $\alpha\in\Omega_c^p(\calF )$, it is easily verified that
$D\left(\sum_i\int_{f_i}\phi_i\alpha\right)$ depends only on the restriction of
$\alpha$ to any neighborhood of the support of $D$ in $M$. Therefore, if
$\zeta\in H_S^p(\calF )$, $\alpha\in\Omega_S^p(\calF )$ is any representative of $\zeta$, and
$\beta\in\Omega_c^p(\calF )$ has the same restriction as $\alpha$ to some neighborhood of $S$,
then $D\left(\sum_i\int_{f_i}\phi_i\beta\right)$ does not depend on the choices of $\alpha$ and
$\beta$, and thus this is a good definition of
$D(\zeta)$.

Theorem~\ref{thm:subfoliations} will follow easily from the following result, which will be proved in
Section~\ref{sec:Poincare}.

\begin{prop}\label{prop:duality}
Let \calF\  be an oriented foliation on a manifold $M$. Let $\iota_1:(M_1,\calF_1) \rightarrow
(M,\calF )$ be a transversely regular immersed oriented subfoliation with
$\codim \calF =\codim \calF_1$, and $f$ a compactly supported basic
function of $\calF_1$. Then there is a class $\zeta\in\calH_c^k(\calF )$,
$k=\dim\calF -\dim\calF_1$,  such that
\begin{equation}\label{e:duality}
\langle(\iota_1,f),(\iota_2,D)\rangle=D(\iota_2^\ast\zeta)
\end{equation}
for any subfoliation $\iota_2:(M_2,\calF_2) \ar (M,\calF)$ and any
holonomy invariant transverse distribution $D$ of $\calF_2$ so that the left
hand side of \eqref{e:duality} is defined. In the right hand side of
\eqref{e:duality}, $D$ is considered as an element of $\calH_S^k(\calF_2)'$ for $S=\supp
D$, and $\iota_2^\ast$ denotes the homomorphism  $\calH_c^k(\calF )\ar{\calH}_S^k(\calF_2)$ induced
by $\iota_2$, which is defined since $\iota_2:S\ar M$ is a proper map. \end{prop}

We do not know whether \eqref{e:duality} completely determines $\zeta$. If so, $\zeta$
could be called the {\em leafwise Poincar\'e dual class} of $(\iota_1,f)$.

\begin{proof}[Proof of Theorem~\ref{thm:subfoliations}] Let $\zeta_m\in\calH_c^k(\calF )$ be the
classes defined by the $(\iota_1,f_m)$ according to Proposition \ref{prop:duality}. If 
$P_n\in\calH_c^k(\calF )'$ is given by the composite
$$\begin{CD}
\calH_c^k(\calF ) @>\iota_2^\ast>> \calH_{S_n}^k(\calF_2)@>D_n>> \R\;,
\end{CD}$$
we have $P_n(\zeta_m)\neq 0$ if and only if $m=n$ by Proposition \ref{prop:duality}, yielding
the linear independence of the $\zeta_m$. 
\end{proof}

\section{Leafwise Poincar\'e duality}\label{sec:Poincare}
This section will be devoted to the proof of Proposition~\ref{prop:duality}.

\subsection{On the Thom class of a vector bundle}
The following lemma is a technical step in the proof of Proposition~\ref{prop:duality},
which will be proved in Section~\ref{subsec:Thom}.

\begin{lem} \label{lemma:Thom}
Let $M$ be a manifold and $\pi:E\ar M$ an oriented vector bundle. Identify $M$ to the image of
the zero section, whose normal bundle is canonically oriented. There is a sequence $\Phi_n$ of
representatives of the Thom class of $E$ such that, if $f$ is any function on
$M$, V is any neighborhood of $M$ in $E$, $K\subset M$ is any compact subset,  and
$\phi:V\ar E$ is any map which restricts to the identity on $M$ and its differential
induces an orientation preserving automorphism of the normal bundle of $M$, then
$\pi^{-1}(K)\cap\phi^{-1}(\supp \Phi_n)$ is compact for large enough $n$, and the
sequence of functions $\int_\pi\phi^\ast(\pi^\ast f\:\Phi_n)$ converges to $f$ over $K$ with
respect to the $\cinf $ topology.\end{lem} 

\begin{cor}\label{cor:Thom}
Let $\pi:E\ar M$ be an oriented vector bundle, and $\iota:N\ar M$ an
immersion. Let $\pi_N:\iota^\ast E\ar N$ be the pull-back vector bundle, and
$\tilde{\iota}:\iota^\ast E\ar E$ the canonical homomorphism. Identify $M$ and $N$ to
the image of the zero sections of $E$ and $\iota^\ast E$, respectively, and consider the
induced orientations on their normal bundles. Let $V$ be an open neighborhood of $N$ in 
$\iota^\ast E$, and $h:V\ar E$ an extension of $\iota$ such that the homomorphism
between the normal bundles of $N$ and $M$, defined by the differential of $h$, restricts to
orientation preserving isomorphisms between the fibers. Let $\Phi_n$ be the forms on $E$ given
by Lemma~\ref{lemma:Thom}, $K\subset N$ a compact subset, and $f$ a function on $M$. Then
$\pi_N^{-1}(K)\cap h^{-1}(\supp \Phi_n)$ is compact for large enough $n$, and the
sequence of functions $\int_{\pi_N}h^\ast(\pi^\ast f\:\Phi_n)$  converge to $\iota^\ast f$
over $K$ with respect to the $\cinf $ topology.
\end{cor}

\begin{proof} Let $U_1,\ldots,U_m$ be a finite open
cover of $K$ such that each $\iota:U_i\ar M$ is an embedding. For each $i$, there is
a compactly supported function $f_i$ on $M$ which is supported in some tubular neighborhood
$W_i$ of $\iota(U_i)$, and such that $f=f_1+\cdots +f_m$ on some neighborhood of $\iota(K)$.
Then, taking a neighborhood $V_i$ of each $U_i$ in $V$ so that $h:V_i\ar E$ is an
embedding, we get 
\begin{equation}\label{e:Thom7} \int_{\pi_N}h^\ast(\pi^\ast f\:\Phi_n)=
\sum_i\int_{\pi_N|_{\pi_N^{-1}(U_i)}}(h|_{V_i})^\ast(\pi^\ast
f_i\:\Phi_n)  \end{equation} 
around $K$, yielding the result if each term in the right hand side of \eqref{e:Thom7}
converges to $\iota^\ast f_i$. Therefore we can assume $\iota$, $\tilde{\iota}$ and $h$ are
embeddings. 

With this assumption, there is an open disk bundle $D$ over $V$, and extensions
$\tilde{\iota}',h':D\ar \pi^{-1}(V)$ of $\tilde{\iota}$ and $h$, respectively, which are
diffeomorphisms onto open subsets of $E$. Let $\phi$ denote the composite
$$\begin{CD}
\tilde{\iota}'(D) @>(\tilde{\iota}')^{-1}>> D@>h'>> \pi^{-1}(V)\;.
\end{CD}$$
Clearly, $\phi$ satisfies the conditions of Lemma \ref{lemma:Thom}, and we can suppose $f$ is
supported in $\tilde{\iota}'(D)$. So
$\int_\pi\phi^\ast(\pi^\ast f\:\Phi_n)$
converges to $f$ over any compact subset of $V$ with respect to the $\cinf $ topology. But 
\begin{eqnarray*}
\iota^\ast\int_\pi\phi^\ast(\pi^\ast
f\:\Phi_n)&=&\int_{\pi_N}\left((\tilde{\iota}')^\ast\phi^\ast(\pi^\ast
f\:\Phi_n)|_V\right)\\
 &=&\int_{\pi_N}h^\ast(\pi^\ast f\:\Phi_n)\;, 
\end{eqnarray*} 
and the result follows. 
\end{proof}

Observe that Lemma~\ref{lemma:Thom} is a particular case of Corollary~\ref{cor:Thom}. The
corollary could be proved directly with the arguments of the lemma, but the notation would
become more complicated.

\subsection{Proof of Lemma~\ref{lemma:Thom}}\label{subsec:Thom}
The following easy observations will be used to prove Lemma~\ref{lemma:Thom}.

\begin{rem}\label{rem:EF}
Let $E$ and $F$ be vector bundles over the manifolds $M$ and $N$, respectively. Suppose 
$f: E\ar F$ is a
homomorphism which restricts to isomorphisms on the fibers, and let $g:M\ar N$ be the map induced by
$f$. Thus the homomorphism $E\ar g^\ast F$, canonically defined by $f$, is an isomorphism.
Therefore there is a composite of homeomorphisms 
$$\cinf(F)\ar \cinf(g^\ast F)\ar \cinf(E)\;.$$
Here, the first homomorphism is canonically
defined by the pull-back diagram of $g^\ast F$, and the second one is induced by the
inverse of $E\ar g^\ast F$. If $s\mapsto s'$ by the above
composite, then $s'$ is determined by  $f(s'(x))=s(g(x))$ for $x\in M$. \end{rem}

\begin{rem}\label{rem:sigma}
Set $E=\R^n\times\R^k$, and let $\pi_i$, $i=1,2$, denote the factor projections
of $E$ onto $\R^n$ and $\R^k$, respectively. Let $K$ be a compact 
subset of $\R^n$, and $\phi:V\ar W$ a diffeomorphism between open neighborhoods
of $\R^n\times\{0\}$. Suppose $\phi$ restricts to the identity on $\R^n\times\{0\}$. For any
$r>0$, let $B_r,S_r\subset\R^k$ respectively denote the Euclidean ball and the Euclidean sphere of
radius r centered at the origin. Then there is an $R>0$ and an open neighborhood $U$ of $K$ such
that, for every $x\in U$ and every $y\in B_R$,  $\{x\}\times\R^k$ intersects transversely
$\phi^{-1}(\R^n\times\{y\})$ at just one point. Moreover, the map  
$$\sigma:U\times B_R\rightarrow(U\times{\R}^k)\cap\phi^{-1}(\R^n\times B_R)\;,$$ 
determined by  
$$\{\sigma(x,y)\}=(\{x\}\times{\R}^k)\cap\phi^{-1}(\R^n\times\{y\})\;,$$ 
is a diffeomorphism. Indeed $\sigma$ is smooth because each
$(U\times\R^k)\cap\phi^{-1}(\R^n\times\{y\})$ can be given as the graph of a map
$\psi_y:U\rightarrow\R^k$ depending smoothly on $y\in B_R$, and 
$\sigma(x,y)=(x,\psi_y(x))$. It also has a smooth inverse since
$(x,y)=(x,\pi_2\,\phi\,\sigma(x,y))$. Therefore, for $r\leq R$, 
$\pi_1:(U\times\R^k)\cap\phi^{-1}(\R^n\times S_r)\ar U$
is a sphere bundle, whose fibers are of volume uniformly bounded by
$Cr^{k-1}$ 
for some $C>0$ if $U$ and $R$ are
small enough.  
\end{rem}

To begin with the proof of Lemma~\ref{lemma:Thom}, fix a Riemannian structure on $E$, and let
$B_r,S_r\subset E$ respectively denote the corresponding open disk bundle and sphere bundle
of radius $r$. Set $S=S_1$. Let $\psi$ be a global angular form of $S$ \cite[\S 11]{BottTu}.
(If $E$ is of rank $k$, $\psi$ is a differential form of degree $k-1$ restricting to unitary
volume forms on the fibers and so that $d\psi=-\pi^\ast e$, where $e$ represents the Euler
class of $S$.) Let $r:E\rightarrow\Bbb R$ denote the radius function, and $h:E\setminus
M\ar S$ the deformation retraction given by $h(v)=v/r(v)$. For each $n$, let also
$\rho_n$ be a function on $[0,\infty)$ such that $-1\leq\rho_n\leq 0$, $\rho_n'\geq 0$,
$\rho_n\equiv -1$ on a neighborhood of $0$, and $\rho_n\equiv 0$ on $[1/n,\infty)$. Then each 
$$\Phi_n=d(\rho_n(r)\,h^\ast\psi)=\rho_n'(r)\,dr\wedge h^\ast\psi-\rho_n(r)\,\pi^\ast e$$
represents the Thom class of $E$ \cite[\S 12]{BottTu}. 

Local orthonormal frames canonically
define  isomorphisms of triviality of $E$ which restrict to local isomorphisms between
restrictions of each $S_r$ and trivial sphere bundles with typical fiber the Euclidean
sphere of radius $r$. So  Remark~\ref{rem:sigma} and the conditions satisfied by $\phi$ 
yield the existence of some $R,C>0$ and some relatively compact open neighborhood $U$ of
$K$ in $M$ so that 
\begin{itemize}
\item $\pi^{-1}(U)\cap\phi^{-1}(B_R)\subset V$,
\item the map $$\phi:\pi^{-1}(U)\cap\phi^{-1}(B_R)\rightarrow\phi\pi^{-1}(U)\cap B_R$$ is a
diffeomorphism whose differential is of fiberwise uniformly bounded norm, and 
\item for $0<r\leq R$, $\phi^{-1}(S_r)$
is transverse to the fibers of $\pi$ over $U$ and $\pi:\pi^{-1}(U)\cap\phi^{-1}(S_r)\rightarrow
U$ is a sphere bundle whose fibers are of volume uniformly bounded by $Cr^{k-1}$. 
\end{itemize}
The $\phi^\ast\Phi_n$ also represent the Thom class of $E$ over $U$ for $n>1/R$. Hence 
\begin{eqnarray} \lefteqn{f-\int_\pi\phi^\ast(\pi^\ast f\:\Phi_n)}\nonumber\\
&=&\int_\pi(\pi^\ast f-\phi^\ast\pi^\ast f)\,\phi^\ast\Phi_n\nonumber\\
&=&\int_0^{1/n}\rho_n'(r)\,dr\,\int_{\pi|_{\pi^{-1}(U)\cap\phi^{-1}(S_r)}}
(\pi^\ast f-\phi^\ast\pi^\ast f)\,\phi^\ast h^\ast\psi\label{e:Thom1}\\
 & &\text{}-\int_\pi(\pi^\ast f-\phi^\ast\pi^\ast f)\,\rho_n(\phi^\ast r)\,\phi^\ast\pi^\ast
e\;. \label{e:Thom2}\end{eqnarray}
We have to prove that \eqref{e:Thom1} and~\eqref{e:Thom2}  
converge uniformly to zero on $K$ as $n\rightarrow\infty$, as well as all of its derivatives
of any order.

Take a Riemannian metric on $M$, and a splitting $ TE={\mathcal V}\oplus{\calH}$, where ${\mathcal V}$ is the
vertical bundle of $\pi$ and ${\mathcal H}$ the horizontal bundle of any Riemannian connection. This
yields a Riemannian structure on $ TE$ defined in the standard way by using the canonical
isomorphisms ${\mathcal V}\cong\pi^\ast E$ and $\calH\cong\pi^\ast TM$. We also have 
$ TM=\calH|_M$, $\calH|_{S_r}\subset TS_r$, and 
\begin{equation}\label{e:Thom3}  TS_r=({\mathcal V}\cap TS_r)\oplus(\calH|_{S_r})\;. \end{equation}
Finally, we can assume  \begin{equation}\label{e:Thom4}
{\mathcal V}\cap\phi_\ast^{-1}(\calH)=0\quad\text{over}\quad\pi^{-1}(U)\cap\phi^{-1}(B_R)
\end{equation}
by the properties of $\phi$.

By the conditions on $\phi$, the supremum of $|\pi^\ast f-\phi^\ast\pi^\ast f|$ over
$\pi^{-1}(U)\cap\phi^{-1}(B_{1/n})$ converges to zero as $n\rightarrow\infty$. Also the
pointwise norm of $\phi^\ast\pi^\ast e$ is uniformly bounded on
$\pi^{-1}(U)\cap\phi^{-1}(B_{1/n})$, thus \eqref{e:Thom2} converges uniformly to zero as
$n\rightarrow\infty$. On the other hand, because the fiberwise norm of each 
$h_\ast: TS_r\rightarrow TS$ is $r^{-1}$, the pointwise norm of $\phi^\ast h^\ast\psi$ is
uniformly bounded on $\pi^{-1}(U)\cap\phi^{-1}(S_r)$ by $C_1r^{-k+1}$ with $C_1>0$ independent
of $r\leq R$. So \eqref{e:Thom1} also converges uniformly to zero on $U$ as
$n\rightarrow\infty$ by the estimate on the volume of the fibers of $\pi$ on
$\pi^{-1}(U)\cap\phi^{-1}(S_r)$.

Now fix vector fields $X_1,\ldots,X_m$ on $U$. By \eqref{e:Thom3} and~\eqref{e:Thom4}
the $X_i$ have liftings $Y_i$ which are sections of $\phi_\ast^{-1}\calH$ over 
$\pi^{-1}(U)\cap\phi^{-1}(B_R)$. For any subset $I\subset\{1,\ldots,m\}$, let $\theta_I$ denote
the composite of Lie derivatives $\theta_{Y_1}\cdots\theta_{Y_l}$ if $I=\{i_1,\ldots i_l\}$
with $i_1<i_2<\cdots<i_l$, and let $\theta_\emptyset$ be the identity homomorphism. Then the
order $m$ derivative $X_1\cdots X_m$ over \eqref{e:Thom1} and~\eqref{e:Thom2} is
respectively  given by
\begin{equation}\label{e:Thom5}
\sum_{I,J}\int_0^{1/n}\rho_n'(r)\,dr\,\int_{\pi|_{\pi^{-1}(U)\cap\phi^{-1}(S_r)}}
\theta_I(\pi^\ast f-\phi^\ast\pi^\ast f)\,\theta_J\phi^\ast h^\ast\psi\;, 
\end{equation}
and
\begin{equation}\label{e:Thom6}
-\sum_{I,J}\int_\pi\theta_I(\pi^\ast f-\phi^\ast\pi^\ast f)\,\rho_n(\phi^\ast 
r)\,\theta_J\phi^\ast\pi^\ast e\;, 
\end{equation}
where $I,J$ runs over the partitions of $\{1,\ldots,m\}$. By the properties of $\calH$ and
$\phi$, the supremum of the $|\theta_I(\pi^\ast f-\phi^\ast\pi^\ast f)|$
on $\pi^{-1}(U)\cap\phi^{-1}(B_{1/n})$ converges to zero as $n\rightarrow\infty$. Hence  
\eqref{e:Thom6} converges uniformly to zero on $K$ because the pointwise norm of the 
$\theta_J\phi^\ast\pi^\ast e$ can be uniformly bounded on $\pi^{-1}(K)\cap\phi^{-1}(B_R)$. The
uniform convergence of \eqref{e:Thom5} to zero  follows by estimating the pointwise norm of
the $\theta_J\phi^\ast h^\ast\psi$ on  $\pi^{-1}(K)\cap\phi^{-1}(S_r)$ by $C_2r^{-k+1}$ for
some $C_2>0$ independent of $r$. This in turn follows by proving a similar estimate for the
pointwise norm of $\theta_J' h^\ast\psi$ on $\phi\pi^{-1}(K)\cap S_r$, where the $\theta_J'$ are
defined in the same way as the $\theta_J$ by using the $Y_i'=\phi_\ast Y_i$ instead of the
$Y_i$. To do this, consider the multiplication map $\mu:[0,R]\times S\ar B_R$. Since
$$\mu_\ast:[0,1]\times(\calH|_S)\subset T([0,1]\times S)\rightarrow\calH$$  restricts
to isomorphisms on the fibers, by Remark~\ref{rem:EF} there are smooth sections $Y_i''$ of
$[0,1]\times(\calH|_S)$ so that $\mu_\ast(Y_i''(r,v))=Y_i'(rv)$. Also because
the composite 
$$\begin{CD}
(0,R]\times S @>\mu>> B_R\setminus M @>h>> S
\end{CD}$$ is the second factor projection, $\mu^\ast
h^\ast\psi$ is the form canonically defined by $\psi$ on $(0,R]\times S$, which
extends smoothly to $[0,R]\times S$. So, if $\theta_J''$ is defined in the same
way as the $\theta_J$ by using the $Y_i''$ instead of the $Y_i$, the pointwise
norm of the $\theta_J''\mu^\ast h^\ast\psi$ is uniformly bounded. Then the
desired estimation of the pointwise norm of the $\theta_J' h^\ast\psi$ follows
by observing that the fiberwise norm of $\mu_\ast:\{r\}\times TS\rightarrow TS_r$ is $r$.  

\subsection{Proof of Proposition~\ref{prop:duality}}
Recall that any local diffeomorphism
$\phi:M\ar N$ induces a homomorphism of complexes,
$\phi_\ast:\Omega_c(M)\rightarrow\Omega_c(N)$, defined as follows. For any
$\alpha\in\Omega_c(M)$, choose a finite open cover $U_1,\ldots,U_n$ of $\supp \alpha$
such that each restriction $\phi:U_i\rightarrow\phi(U_i)$ is a diffeomorphism. There is a
decomposition $\alpha=\alpha_1+\cdots+\alpha_n$ so that $\supp \alpha_i\subset U_i$. 
For each $i$, there is a unique $\beta_i\in\Omega_c(N)$ supported in $\phi(U_i)$ such that
$\beta_i|_{\phi(U_i)}$ corresponds to $\alpha_i|_{U_i}$ by $\phi$. Define
$\phi_\ast\alpha=\beta_1+\cdots+\beta_n$. This definition is easily checked to be independent of
the choices involved and compatible with the differential maps. If
$\phi:(M,\calF)\rightarrow(N,\calG )$ is a local diffeomorphism of foliated manifolds, we
similarly have a homomorphism $\phi_\ast:\Omega_c(\calF )\rightarrow\Omega_c(\calG )$ which is
compatible with the leafwise de~Rham derivative. Moreover $\phi_\ast$ is surjective if so is
$\phi$.

Now Proposition~\ref{prop:duality} can be proved as follows. 

There is a canonical injection of $ T\calF_1$  as vector subbundle of $\iota_1^\ast T\calF$. Let  $E=\iota_1^\ast T\calF / T\calF_1$, and
$\pi:E\ar M_1$ the bundle projection. Identify $M_1$ with the image of the zero section of $E$.
Fixing any Riemannian metric on $M$, there are induced Riemannian metrics on the $M_i$, and an
induced Riemannian structure on $E$. For each $r>0$, let $B_r\subset E$ denote the open disk bundle
of  radius $r$ over $M_1$. Then there is an $R>0$ and an open neighborhood $U$
of the support of $f$ in $M_1$ such that, if $V=\pi^{-1}(U)\cap B_R$, the
restriction of  $\iota_1$ to $U$ can be extended to a map of foliated manifolds, 
$\tilde{\iota}_1:(V,\pi^\ast\calF_1|_V)\rightarrow(M,\calF )$, which is defined over
each $x\in M_1$ as a composite of the restriction of the canonical homomorphism
$$\left(\iota_1^\ast  T\calF / T\calF_1\right)_x\rightarrow 
 T_{\iota_1(x)}\calF /\iota_{1\ast} T_x\calF_1\equiv 
(\iota_{1\ast} T_x\calF_1)^\perp\cap T_x\calF\;,$$ and the  exponential map of the
leaves of  \calF\  defined on the ball of radius $R$ centered at zero in 
$(\iota_{1\ast} T_x\calF_1)^\perp\cap T_x\calF$.  By elementary
properties of the exponential map and since $\iota_1$ is transversely regular
with $\codim \calF_1=\codim \calF $, $R$ can be chosen so that $\tilde{\iota}_1$ is a
local diffeomorphism and $\tilde{\iota}_1^\ast\calF =\pi^\ast\calF_1|_V$. 

$E$ is of rank $k$, and with an induced orientation. The representatives
$\Phi_n$ of its Thom class, given by Lemma~\ref{lemma:Thom}, can be assumed to be supported
in $B_R$. The $\Phi_n$ are of degree $k$, closed and compactly supported in the vertical
direction, i.e. with compactly supported restrictions to the fibers. Moreover all the
$\Phi_n|_{B_R}$ are pairwise cohomologous in the complex of forms in
$\Omega^\cdot(B_R)$ which are compactly supported in the vertical direction. On
the other hand, $f$ is basic and compactly supported. So the $\pi^\ast
f\,\Phi_n$ restrict to leafwise closed forms
$\alpha_n\in\Omega_c^k(\pi^\ast\calF_1|_V)$ which are pairwise cohomologous. Thus
the  $\tilde{\iota}_{1\ast}\alpha_n\in\Omega_c^k(\calF)$ are leafwise closed and
define the same class $\zeta\in\calH_c^k(\calF )$.

Let $U_1,\ldots,U_m$, be an open cover of the support of $f$ in $U$ such that each 
$\iota_1:U_j\ar M$ is an embedding, $j=1,\ldots,m$. The above $R$ can be chosen small enough 
so that the
$\tilde{\iota}_1:V_j=\pi^{-1}(U_j)\cap B_R\rightarrow\tilde{\iota}_1(V_j)$ are diffeomorphisms.
Take a decomposition $f=f_1+\cdots+f_m$ with each $f_j$ compactly supported in $U_j$, and let
$\alpha_{n,j}\in\Omega_c^k(\pi^\ast\calF_1|_{V_j})$ be the restriction to the leaves
of $\pi^\ast f_j\:\Phi_n$. Then, by definition, 
$$\tilde{\iota}_{1\ast}\alpha_n=\beta_{n,1}+\cdots+\beta_{n,m}\;,$$ where each
$\beta_{n,j}\in\Omega_c^k(\calF )$ is the extension by zero of the forms in
$\Omega_c^k(\calF |_{\tilde{\iota}_1(V_j)})$ which correspond to the
$\alpha_{n,j}|_{V_j}$ by $\tilde{\iota}_1$.

Given $\iota_2:(M_2,\calF_2)\rightarrow(M,\calF )$, we use the same notation as in the 
preamble of Definition~\ref{defn:intersection}. We can clearly assume the $U_j$ are contained in
$N_1$. Let $\iota=\iota_1\sigma_1=\iota_2\sigma_2:T\ar M$. There is a canonical isomorphism
$$\iota^\ast T\calF \cong
\sigma_1^\ast T\calF_1\oplus\sigma_2^\ast T\calF_2$$
because the $\iota_i|_{N_i}$ intersect transversely in \calF. So
$\sigma_1^\ast E\cong\sigma_2^\ast T\calF_2$
canonically. This isomorphism will be considered as an identity.

Let $\pi_T:\sigma_1^\ast E\ar T$ be the pull-back vector
bundle projection, and $\tilde{\sigma}_1:\sigma_1^\ast E\ar E$ the canonical
homomorphism. Identify $T$ to the image of the zero section of
$\sigma_1^\ast E$. For each $j$, take a relatively compact open subset
$O_j\subset\sigma_1^{-1}(U_j)$ containing the compact set 
$\supp \sigma_1^\ast f_j\cap\supp D_T$. The above $R$ can be chosen small enough 
so that $\sigma_2:O_j\ar N_2$ has an extension to a local diffeomorphism  
$\tilde{\sigma}_2:\pi_T^{-1}(O_j)\cap\tilde{\sigma}_1^{-1}(B_R)\ar N_2$ defined as the
composite of the restriction of the canonical homomorphism $\sigma_1^\ast
E\equiv\sigma_2^\ast T\calF_2\rightarrow T\calF_2$, and the exponential map
of the leaves of $\calF_2$ defined on the tubular neighborhood of radius $R$ of the zero
section in $ T\calF_2$. In this way, $\tilde{\sigma}_2$ maps each fiber of $\pi_T$ into
a leaf  of $\calF_2$. Observe that the diagram 
$$\begin{CD}
\tilde{\sigma}_1^{-1}(B_R)\cap\pi_T^{-1}(O_j)    @>\tilde{\sigma}_1>>   V_j\\
@V{\tilde{\sigma}_2}VV                                                @VV\tilde{\iota}_1V\\ 
N_2                                              @>\iota_2>>              M   
\end{CD}$$
is obviously non-commutative in general. This is the main technical difficulty. To solve it, we
have chosen the $\Phi_n$ so that their supports concentrate around $M_1$ and satisfy the needed
properties at the limit (Lemma~\ref{lemma:Thom} and Corollary~\ref{cor:Thom}).

We need the observation that  
\begin{equation}\label{e:pull-back}
\sigma_2\sigma_1^{-1}(A)=(\iota_2|_{N_2})^{-1}\iota_1(A) 
\end{equation}
for any subset $A\subset N_1$, as can be easily checked. 

Using the
compactness of $\overline{B_R}\cap\pi^{-1}(\supp f_j)$ and since 
$$\supp f_j=\bigcap_{0<r<R}\overline{B_r}\cap\pi^{-1}(\supp f_j)\;,$$
we easily get 
$$ \iota_1(\supp f_j)=
\bigcap_{0<r<R}\tilde{\iota}_1\left(\overline{B_r}\cap\pi^{-1}(\supp f_j)\right)\;.$$
Therefore
\begin{align*}
\bigcap_{0<r<R}\supp D\cap\iota_2^{-1}
\tilde{\iota}_1\left(\overline{B_r}\cap\pi^{-1}(\supp f_j)\right) 
&=\supp D\cap\iota_2^{-1}\iota_1(\supp f_j)\\ 
&=\supp D\cap\sigma_2\sigma_1^{-1}(\supp f_j)\\ 
&=\sigma_2(\sigma_2^{-1}(\supp D)\cap\sigma_1^{-1}(\supp f_j))\\ 
&=\sigma_2(\supp D_T\cap\supp \sigma_1^\ast f_j)\;, 
\end{align*} 
where the second equality follows by \eqref{e:pull-back}. Then, since the $$\supp D\cap\iota_2^{-1}
\tilde{\iota}_1\left(\overline{B_r}\cap\pi^{-1}(\supp f_j)\right)$$ are compact, and 
since  $$\tilde{\sigma}_2\left(\tilde{\sigma}_1^{-1}(B_R)\cap\pi_T^{-1}(O_j)\right)$$ is an open
neighborhood of 
$$\sigma_2(\supp D_T\cap\supp \sigma_1^\ast f_j)\;,$$ 
there is an $r<R$ such that
$$\supp D\cap\iota_2^{-1}\tilde{\iota}_1\left(\overline{B_r}\cap\pi^{-1}(\supp f_j)\right)\subset 
\tilde{\sigma}_2\left(\tilde{\sigma}_1^{-1}(B_R)\cap\pi_T^{-1}(O_j)\right)\;.$$
So 
$$\supp D\cap\supp \iota_2^\ast\beta_{n,j}\subset\tilde{\sigma}_2(W_j)$$
for large enough $n$, where
$$W_j=\tilde{\sigma}_2^{-1}\iota_2^{-1}\tilde{\iota}_1(V_j)\cap
\tilde{\sigma}_1^{-1}(B_R)\cap\pi_T^{-1}(O_j)\;.$$
We can assume this holds for every $n$. 
Hence there is some $\omega_{n,j}\in\Omega_c^k(\calF_2)$ which is supported in
$\tilde{\sigma}(W_j)$ and has the same restriction to some neighborhood of $\supp D$ as
$\iota_2^\ast\beta_{n,j}$. If $\calF_{\pi_T}$ is the foliation on $\sigma_1^\ast E$ 
defined by the fibers of $\pi_T$, there is some
$\gamma_{n,j}\in\Omega_c^k(\calF_{\pi_T}|_{W_j})$ such that 
$(\tilde{\sigma}_2|_{W_j})_\ast\omega_{n,j}=
\iota_2^\ast\beta_{n,j}$. We get 
\begin{equation}\label{e:gammanj}
D(\iota_2^\ast\zeta)=
\sum_jD_T\left(\int_{\pi_T|_{W_j}}\gamma_{n,j}\right)
\end{equation} 
by definition. 
Let $h_j:W_j\ar E$ be the immersion given by the composite
$$\begin{CD}
W_j @>\tilde{\sigma}_2>> \iota_2^{-1}\tilde{\iota}_1(V_j)
@>\iota_2>> \tilde{\iota}_1(V_j)
@>\tilde{\iota}_1^{-1}>> V_j\subset E\;.
\end{CD}$$
Clearly $h_j$ is an extension of $\sigma_1:O_j\ar N_1\subset M_1\subset E$, and 
$\gamma_{n,j}=h_j^\ast(\pi^\ast f_j\:\Phi_n)$ around 
$W_j\cap\pi_T^{-1}(\sigma_1^{-1}(\supp f_j)\cap\supp D_T)$. Moreover the
homomorphism between the normal bundles of $O_j$ and $M_1$, defined by the differential of
$h_j$, restricts to isomorphisms on the fibers. These isomorphisms are orientation preserving on
fibers over points with $\varepsilon=1$, and orientation reversing on fibers over points
with $\varepsilon=-1$. Therefore $\int_{\pi_T|_{W_j}}\gamma_{n,j}$ converges to 
$\varepsilon\,\sigma_1^\ast f_j$ on $\sigma_1^{-1}(\supp f_j)\cap\supp D_T$ 
with respect to the $\cinf $ topology by Corollary~\ref{cor:Thom}. 
Hence \eqref{e:gammanj} is equal to $\langle(\iota_1,f),(\iota_2,D)\rangle$, and the
proof is complete.

\begin{rem}\label{rem:subfoliations}
Observe that, in Proposition~\ref{prop:duality}, $\zeta$ has representatives 
supported in any neighborhood of
$\iota_1(M_1)$. Thus, in Theorem~\ref{thm:subfoliations}, the linearly independent
classes $\zeta_m\in\calH_c^\cdot(\calF )$ also have representatives supported in any neighborhood
of $\iota_1(M_1)$. \end{rem}

\section{Case where the leaves have homology classes with non-trivial intersection
\label{sec:homology}} 
This section will be devoted to the proof of Corollary~\ref{cor:homology1}. 

Let $M$ be the ambient manifold of \calF. Let $\pi_i:\iota_i^\ast( TM/ T\calF )\ar K_i$
denote the pull-back vector bundle projection, and identify $K_i$ to the image of its zero
section. Fix a Riemannian metric
on $M$ and, for some $R>0$, let $M_i$ be the tubular neighborhood of radius $R$ around $K_i$ in 
$\iota_i^\ast( TM/ T\calF)$. Such $R$ can be chosen so that the maps 
$\tilde{\iota}_i:M_i\ar M$ are well defined as composites of the restrictions
of the canonical homomorphisms 
$\iota_i^\ast( TM/ T\calF)\ar 
 TM/ T\calF\equiv( T\calF)^\perp$, and the restriction of the exponential
map of $M$ to the tubular neighborhood of radius $R$ of the zero section of
$ T\calF^\perp$. Choose $R$ small enough so that
$\tilde{\iota}_i:\pi_i^{-1}(x_i)\cap M_i\rightarrow\tilde{\iota}_i(\pi_i^{-1}(x_i)\cap M_i)$ is
an embedded transversal of \calF\  for each $i$ and each $x_i\in M_i$. Observe that
$\tilde{\iota}_1(\pi_1^{-1}(x_1)\cap M_1)=\tilde{\iota}_2(\pi_2^{-1}(x_2)\cap M_2)$ if
$\iota_1(x_1)=\iota_2(x_2)$. The $\tilde{\iota}_i$ are thus transverse to \calF, 
and the $\calF_i= \tilde{\iota}_i^\ast\calF$ have the same codimension as \calF. Then
$\tilde{\iota}_i$ are transversely regular immersions of foliated manifolds. By deforming the 
$\iota_i$ if needed, we can suppose the $\iota_i$ intersect each other transversely, and thus
the $\tilde{\iota}_i$ intersect transversely in \calF. Moreover the orientations of the $K_i$
induce  orientations of the $\calF_i$.

Each $K_i$ is a closed leaf of $\calF_i$ whose holonomy representation is given by the
composite \eqref{e:composite}. So the holonomy group of $K_i$ is given by germs of
local isometries. Hence $\calF_i$ is a Riemannian foliation around $K_i$, as follows easily
from \cite[Theorem~2 in Chapter~IV]{Haefliger58}. (See also \cite[Theorem~2.29 in
Chapter~II]{Godbillon} or \cite{HectorHirsch83}.) We can assume the whole $\calF_i$ is Riemannian, which
can thus be described as follows \cite{Haefliger88,MolinoPierrot}. Fix a transverse Riemannian structure
of $\calF_i$. Let $Q_i$ be the $O(q)$-principal bundle over $M_i$ of transverse orthonormal
frames of $\calF_i$ with the transverse Levi-Civita connection, and $\hat\calF_i$ the
horizontal lifting of $\calF_i$ to $Q_i$ \cite{Molino82,Molino88}. Let $P_i$ be a leaf closure
of $\hat\calF_i$ over $K_i$. Then $P_i$ is an $H_i$-principal bundle over $K_i$ for some
closed subgroup $H_i\subset O(q)$. For the open disk $B\subset\R^q$ of radius $R$ 
centered at the origin, we can assume 
$M_i\equiv(P_i\times B)/H_i$ as fiber bundles over $K_i$, where the $H_i$-action on
$P_2\times B$ is the diagonal one; i.e. $(z,v)h=(zh,h^{-1}v)$ for $(z,h)\in P_i\times B$ and
$h\in H_i$. Moreover the above identity can be chosen so that $\calF_i$ is identified to the
foliation whose leaves are the projections of products of leaves of $\hat\calF_i$ in $P_i$
and points in $B$. (This description is simpler than the one in \cite{Haefliger88} and
\cite{MolinoPierrot} because the leaf closure $K_i$ is just a compact leaf.)

Consider the transverse 
Riemannian structure of each $\calF_i$ defined by the Euclidean metric on $B$ using the above
description. Since the elements in the image of the composites \eqref{e:composite} are germs of
local isometries for the same metric on $\R^q$, the composite
$$
B\equiv\pi_2^{-1}(x_2)\cap M_2 \stackrel{\tilde{\iota}_2}{\longrightarrow}
\tilde{\iota}_2(\pi_2^{-1}(x_2)\cap M_2)
=\tilde{\iota}_1(\pi_1^{-1}(x_1)\cap M_1)\stackrel{\tilde{\iota}_1^{-1}}{\longrightarrow}
\pi_1^{-1}(x_1)\cap M_1\equiv B
$$  
is an isometry
around the origin
for all $(x_1,x_2)\in K_1\times K_2$ with $\iota_1(x_1)=\iota_2(x_2)$. We can
assume such composite is an isometry on the whole $B$, which will be denoted by
$\phi_{x_2,x_1}$.

With the above description, any compactly supported basic function $f$ of $\calF_1$ can be
canonically considered as an $H_1$-invariant compactly supported function on $B$, and any 
compactly supported holonomy invariant transverse distribution $D$ of $\calF_2$ can be
canonically considered as a compactly supported $H_2$-invariant distribution on $B$. 
For such $f$ and $D$, we
clearly have 
\begin{equation}\label{e:homology}
\langle(\tilde{\iota}_1,f),(\tilde{\iota}_2,D)\rangle=
\sum\varepsilon(x_1,x_2)\,D(\phi_{x_2,x_1}^\ast f)\;,\end{equation} 
where the sum runs over the 
pairs $(x_1,x_2)\in  K_1\times K_2$ with $\iota_1(x_1)=\iota_2(x_2)$. Here 
$\varepsilon(x_1,x_2)=\pm 1$ depending on whether the identity 
$$ T_{\iota_i(x_i)}B\equiv
\iota_{1\ast} T_{x_1}K_1\oplus\iota_{2\ast} T_{x_2}K_2$$ is orientation preserving or
orientation reversing. Let $f_m$ be a sequence of compactly supported $O(q)$-invariant
functions in $B$ with integral equal to one and pairwise disjoint supports, and let $\mu_m$ 
be the restriction of the Euclidean measure to the support of $f_m$. Then 
$$\langle(\tilde{\iota}_1,f_m),(\tilde{\iota}_2,\mu_n)\rangle=
\langle\iota_1,\iota_2\rangle\,\int_Bf_m\,d\mu_n$$
by \eqref{e:homology}, where $\langle\iota_1,\iota_2\rangle$ is the intersection number of
$\iota_1$ and $\iota_2$ in $B$. So $\dim\calH_c^{k_2}(\calF )=\infty$ by 
Theorem~\ref{thm:subfoliations}. Similarly, $\dim\calH_c^{k_1}(\calF )=\infty$, which completes
the proof.

\section{Case of suspension foliations \label{sec:suspension}}

\begin{proof}[Proof of Corollary~\ref{cor:suspension1}] 
Recall the notation used for suspension foliations in the statement of
Corollary~\ref{cor:suspension1}, and consider the fiber bundles $M_i = \iota_i^{\ast}M$ over $K_i$. Each
canonical map $\tilde{\iota}_i:M_i\ar M$ is transverse to 
\calF,  and let $\calF_i=\tilde{\iota}_i^\ast\calF$.  Then
$\tilde{\iota}_i$ are transversely regular immersions of foliated manifolds. By deforming the $\iota_i$ if needed,
we can suppose the $\iota_i$ intersect each other transversely, thus the $\tilde{\iota}_i$ 
intersect each other transversely in \calF. Moreover the orientations of the $K_i$ induce 
orientations of the $\calF_i$. 

The group of deck transformations of each pull-back covering map $\iota_i^\ast L\rightarrow
K_i$ is isomorphic to
$\G_i$, and $\calF_i$ is canonically isomorphic to the corresponding 
suspension foliation given by the restriction to $\G_i$ of the $\G$-action on 
$T$. Hence the $f_m$ can be canonically 
considered as compactly supported basic functions of $\calF_1$, and the $D_n$ can be
canonically considered as holonomy invariant transverse distributions of $\calF_2$. The
$\tilde{\iota}_2:\supp D_n\ar M$ are clearly proper, and we easily get 
$$\langle(\tilde{\iota}_1,f_m),(\tilde{\iota}_2,D_n)\rangle=
\langle\iota_1,\iota_2\rangle\,D_n(f_m)\;.$$
Therefore the result follows from Theorem~\ref{thm:subfoliations}. 
\end{proof}

\begin{proof}[Proof of Corollary~\ref{cor:suspension2}] Let $A_n$ be a sequence of
$\G_1$-saturated open subsets of $T$ with  non-trivial $\mu$-measure and pairwise disjoint
$\G_2$-saturations. Clearly, there are open sets $B_n$ of $T$ with positive
$\mu$-measure and such that $\overline{B_n}\subset A_n$. Since the closure of
$\G_1$ in the group of diffeomorphisms of $T$ is a compact Lie group, there exists a
sequence of non-negative $\G_1$-invariant functions $f_n$ on $T$  such that
$B_n\subset\supp f_n\subset A_n$. Let $\mu_n$ be the $\G_2$-invariant measure
on $T$ defined as the product of $\mu$ and the characteristic
function of the closure of the $\G_2$-saturation of $\supp f_n$. Then 
$\int_Tf_m\,d\mu_n\neq 0$ if and only if $m=n$, 
and the result follows by Corollary~\ref{cor:suspension1}. 
\end{proof}

\begin{proof}[Proof of Corollary~\ref{cor:suspension3}] Since some lift of $c$ to $L$ joins two distinct
points of its end set,
$L$ is disconnected by some codimension one immersed closed submanifold, $\iota:K\ar L$, such that $c$
and $\pi\iota$ define homology classes of $B$ with non-trivial intersection. Clearly, the
composite
$$\begin{CD}
\pi_1(K) @>\pi_1(\pi\iota)>>\pi_1(B) @>h>> \G
\end{CD}$$
is trivial, and the image of the composite
$$\begin{CD}
\pi_1(S^1) @>\pi_1(c)>>\pi_1(B) @>h>> \G
\end{CD}$$
is $\langle a\rangle$. Take a sequence $A_n$ of disjoint non-trivial $H$-invariant open subsets of $T$.
Since $H$ is an abelian compact Lie group (a torus), there
is an $H$-invariant probabilistic measure supported in any $H$-orbit in $T$. Take thus one of such measures
$\mu_n$ supported in each $A_n$. Then the result follows from Corollary~\ref{cor:suspension2} by
taking as $\mu$ the sum of the $\mu_n$. 
\end{proof}

To prove Proposition~\ref{prop:orbit density}, we use the following.

\begin{lem} \label{lemma:orbit density} Let $\G$ be a finitely
generated group, and $X$  a connected $T_1$ topological space. For any continuous action of
$\G$ on $X$, a finite union of orbits is dense if and only
if so is each orbit in the union. \end{lem} 

\begin{proof} Take $x_1,\ldots,x_n\in X$ such that 
$$X=\overline{\G x_1\cup\ldots\cup\G x_n}=\overline{\G x_1}\cup\ldots
\cup\overline{\G x_n}\;.$$
Each orbit closure $\overline{\G x_i}$ can be decomposed as a disjoint union
of sets
$$L_i=\bigcap_F\overline{(\G\setminus F)x_i}\;,\quad I_i=\overline{\G x_i}\setminus L_i\;,$$
where $F$ runs over the finite subsets of $\G$. We have $X=L\cup I$, where
$L=\bigcup_{i=1}^nL_i$ and $I=\bigcup_{i=1}^nI_i$. Moreover, since $L$ is
saturated we have $L\cap I=\emptyset$. So $I=\emptyset$ because $X$ is $T_1$ and
connected. (If we had $I\neq\emptyset$, for any $y\in I$, $\{y\}$ would be closed in
$X$ because $X$ is $T_1$. But since $L$ is closed and $I=X\setminus L$
is discrete, $\{y\}$ would be also open in $X$. Thus $X$ would not be connected.)
Therefore $X=L$ and $L_i= \overline{\G x_i}$ for each $i$. But each $L_i$ is
closed in $X$, and $L_i\cap L_j\neq\emptyset$ implies $L_i=L_j$,
obtaining $X=L_i$ for every $i$ by the connectedness of $X$. 
\end{proof}

\begin{proof}[Proof of Proposition~\ref{prop:orbit density}] Clearly, if the $C$-orbits are dense in
$X$, so are the $\G$-orbits. 

Reciprocally, suppose the $\G$-orbits are dense. By a theorem of Stallings 
\cite{Stallings}, there is a finite normal subgroup $F\subset\G$ such that
$\G_1=\G/F$ is isomorphic either to $\Bbb Z$ or to the diedric group $\Z_2\ast\Z _2$. The
action of $\G$ on $X$ defines an action of $\G_1$ on the connected $T_1$ space $X_1=X/F$
with dense orbits. Since $C$ is infinite, so is its projection $C_1$ to $\G_1$, and any
infinite subgroup of such $\G_1$ is of finite index. Therefore any $\G_1$-orbit in $X_1$
is a finite union of $C_1$-orbits, and thus the $C_1$-orbits are dense in $X_1$ by
Lemma~\ref{lemma:orbit density}. This implies the density of the $CF$-orbits in $X$ because the
canonical projection of $X$ onto $X_1$ is open and continuous. But any $CF$-orbit is a finite
union of $C$-orbits. Hence the $C$-orbits are dense by Lemma~\ref{lemma:orbit density}. 
\end{proof}

\section{Case of Lie foliations with compact semisimple structural
Lie algebra \label{sec:css}}

Theorem~\ref{thm:css} will be proved in this section (Corollaries~\ref{cor:css,k=1,2}
and~\ref{cor:css}). 

\subsection{Construction of a spectral sequence for an arbitrary Lie foliation on a closed
manifold}  

Let \calF\  be a Lie foliation with dense leaves on a closed manifold $M$. Let $\g$ be the 
structural Lie algebra of \calF,  and $G$ the simply connected Lie group with Lie algebra $\g$.
Let $\pi:\widetilde{M}\ar M$ be the universal covering map. Then the leaves of
$\widetilde{\calF}=\pi^\ast\calF$ are the fibers of a fiber bundle
$D:\widetilde{M}\ar G$. It will be convenient to consider the right action of $\pi_1(M)$ on $\widetilde M$
by deck transformations and the left action of $G$ on itself by left translations. Thus $D$ is
anti-equivariant with respect to the holonomy homomorphism
$h:\pi_1(M)\ar G$; i.e., $D(\tilde{x}\gamma)=h(\gamma)^{-1}D(\tilde{x})$ \cite{Fedida}. The density of
the leaves implies the density of $\G =
h(\pi_1(M))$ in $G$. The
homomorphism $h$ defines an action of $\pi_1(M)$ on $G$ by left translations, yielding
the corresponding suspension foliation $\calG $ on
$N=\left(\widetilde{M}\times G\right)/\pi_1(M)$ (defined as in Section~\ref{sec:suspension}). $\calG $
is a Lie foliation with the same transverse structure as \calF,  given by $(G,\G)$.

The section $(\text{id} ,D):\widetilde{M}\rightarrow\widetilde{M}\times G$ is
$\pi_1(M)$-equivariant:
$$(\text{id} ,D)(\tilde{x}\gamma)=(\tilde{x}\gamma,D(\tilde{x}\gamma))= 
(\tilde{x}\gamma,\gamma^{-1}D(\tilde{x}))=(\tilde{x},D(\tilde{x}))\gamma\;.$$
Thus $(\text{id} ,D)$ defines a section $s:M\ar N$, and $N$ is trivial as
principal $G$-bundle over $M$. Clearly $s$ is transverse to $\calG $, and
$s^{\ast}\calG =\calF$.

Let $\overline{D}:\widetilde{M}\times G\ar G$ be defined by
$\overline{D}(\tilde{x},g)=g^{-1}D(\tilde{x})$. Such 
$\overline{D}$ is $\pi_1(M)$-invariant:
$$\overline{D}((\tilde{x},g)a)=\overline{D}(\tilde{x}a,h(a)^{-1}g)=
g^{-1}h(a)D(\tilde{x}a)=g^{-1}D(\tilde{x})\;.$$
So $\overline{D}$ defines a map $D_N:N\ar G$.
Clearly $D_Ns=\text{const}_e$, where $e$ is the identity
element in $G$. Moreover $D_N$ is $G$-anti-equivariant:
$$D_N([\tilde{x},g]g')=D_N([\tilde{x},gg'])=
(gg')^{-1}D(\tilde{x})=(g')^{-1}D_N([\tilde{x},g])\;.$$
Therefore $D_N$ is the composite of the second factor projection of the 
trivialization of $N\ar M$ defined by $s$ and the
inversion map on $G$.                       

Let $\widetilde{\calF}$ also denote the foliation on $N$ defined by the lifting of
\calF\  to all the leaves of $\calG $. $\widetilde{\calF}$ is a subfoliation 
of $\calG $ whose leaves are the intersections of the leaves of $\calG $
with all the translations of $s(M)$.

Let $\nu\subset T\calG $ be a $G$-invariant subbundle so that $ T\calG =\nu\oplus T\widetilde{\calF}$. We
get
$$\bigwedge T\calG ^{\ast}=\bigwedge\nu^{\ast}\otimes\bigwedge T\widetilde\calF^\ast\;,$$
and thus there is a bigrading of $\Omega=\Omega(\calG )$ defined by 
$$
\Omega^{u,v}=\cinf\left(\bigwedge^u\nu^{\ast}\otimes\bigwedge^v T\widetilde\calF^{\ast}\right)\;,
\quad u,v\in\Z\;.
$$
For simplicity, $d_{\calG }$ will be denoted by $d$. There is a decomposition of $d$ 
as sum of bihomogeneous components, $d=d_{0,1}+d_{1,0}+d_{2,-1}$, where each double
subindex denotes the corresponding bidegree. From $d^2=0$ we get
\begin{equation}
d_{0,1}^2=d_{2,-1}^2=d_{0,1}d_{1,0}+d_{1,0}d_{0,1}=0\;,
\label{e:di,j 1} \end{equation}
\begin{equation}\label{e:di,j 2} 
d_{1,0}d_{2,-1}+d_{2,-1}d_{1,0}=d_{1,0}^2+d_{0,1}d_{2,-1}+d_{2,-1}d_{0,1}=0\;.
\end{equation}     
The decreasing filtration of $(\Omega,d)$ by the differential ideals
\begin{equation}\label{e:filtration} F^l=\Omega^{l,\cdot}\wedge\Omega\;,
\end{equation}
depends only on $\left(\calG ,\widetilde{\calF}\right)$; it could be defined without
using $\nu$. So we get a spectral sequence $(E_i,d_i)$ which converges to $H^\cdot(\calG )$. As
for the spectral sequence of a foliation (see e.g. \cite{Alv1}), in this case there are canonical
identities  \begin{equation}\label{e:E0,E1}
(E_0,d_0)\equiv(\Omega,d_{0,1})\;,\ (E_1,d_1)\equiv(H(\Omega,d_{0,1}),d_{1,0\ast})\;.
\end{equation}
The $\cinf $ topology on the space of differential forms induces a topology on
each $E_i$ which is not Hausdorff in general.

At each $z\in N$ we have 
$$\begin{CD}
D_{N\ast}:\nu_z @>\cong>> T_{D_N(z)}G\;.
\end{CD}$$ 
So for each $X\in\g$ there is a well defined vector
field $X^{\nu}\in\cinf(\nu)$ which is $D_N$-projectable and such that
$D_{N\ast}X^{\nu}=X$. Such $X^{\nu}$ is $G$-invariant since $X^{\nu}_zg\in\nu_{zg}$ and
$$D_{N\ast}(X^{\nu}_zg)=
g^{-1}D_{N\ast}X^{\nu}_z=g^{-1}X_{D_N(z)}=X_{g^{-1}D_N(z)}=X_{D_N(zg)}\;.$$
Let $\theta_X$ and $i_X$ respectively denote the Lie derivative and interior
product on $\Omega$ with respect to $X^{\nu}$. (We are considering $\theta_X$ and
$i_X$ as operators on the leaves of $\calG $, but preserving
smoothness on $N$.) By comparing bidegrees in the usual formulas that relate Lie
derivatives, interior products and the de~Rham derivative, we easily get
$$d_{0,1}i_X+i_Xd_{0,1}=0\;,$$
$$(\theta_X)_{0,0}d_{0,1}=d_{0,1}(\theta_X)_{0,0}\;,$$
$$i_{[X,Y]}=(\theta_X)_{0,0}i_Y-i_Y(\theta_X)_{0,0}\;,$$
$$(\theta_X)_{0,0}=d_{1,0}i_X+i_Xd_{1,0}\;,$$
$$(\theta_{[X,Y]})_{0,0}=(\theta_X\theta_Y-\theta_Y\theta_X)_{0,0}-
d_{0,1}i_{\Xi(X\wedge Y)}-i_{\Xi(X\wedge Y)}d_{0,1}\;,$$
where $\Xi:\bigwedge^2\g \ar \cinf\left( T\widetilde{\calF}\right)$ is given by
$$\Xi(X\wedge Y)=[X^{\nu},Y^{\nu}]-[X,Y]^{\nu}\;.$$
Therefore we get the operation $(\g ,i_1,\theta_1,E_1,d_1)$, where
$i_{1X}\equiv i_{X\ast}$ and $\theta_{1X}\equiv(\theta_X)_{0,0\ast}$ according to
\eqref{e:E0,E1}, and the algebraic connection $D_N^{\ast} : {\g}^{\ast}\ar
E_1^{1,0}\subset \Omega^{1,0}$ \cite{GHV:III}. Then $$E_2^{u,v}\cong H^u(\g ;\ \theta_1:\g \ar
\End(E_1^{0,v}))\;.$$

Let $\phi:N\times\g \ar N$ be defined by $\phi(z,X)=
X_1^{\nu}(z)$, where $X_t^{\nu}$ denotes the uniparametric group of
transformations defined by $X^{\nu}$, considered as group of transformations of the leaves
of $\calG $ preserving smoothness on $N$. Then the following
diagram is commutative
$$\begin{CD}
N\times\g          @>\phi>>            N\\
@V{D_N\times\exp}VV                    @VV{D_N}V\\
G\times G          @>>>                G\;,
\end{CD}$$
where the lowest map denotes the operation on $G$. (This follows because
$X_t=R_{\exp(tX)}$ for all $X\in\g$.)

\subsection{Tensor product decomposition of $E_2$ when $\g$ is compact semisimple}
From now on suppose $\g$ is compact semisimple, and thus $G$ is compact
\cite{Poor}.

\begin{thm}\label{thm:E2}
With the above notations, $$E_2^{u,v}\cong H^u(\g )\otimes E_2^{0,v}=
H^u(\g )\otimes (E_1^{0,v})_{\theta_1=0}\;.$$
\end{thm}

The result follows with the same type of arguments as in those given in
Sections~2 and~3 of \cite{Alv5} to prove Theorem 3.5 in \cite{Alv5}. We
will indicate the main steps in the proof because some of them will be needed later.

Consider the canonical biinvariant metric on $G$ \cite[Chapter 6]{Poor}, and
let $C\subset G$ and $C^{\ast}\subset\g$ be the cut locus and tangential cut
locus corresponding to the identity element $e\in G$. Let $B^{\ast}$ be the
radial domain in $\g$ bounded by $C^{\ast}$, and let $B=
\exp(B^{\ast})$. From the general properties of the cut locus
we have $C=\partial B=G\backslash B$,
$\exp:B^{\ast}\ar B$ is a diffeomorphism, $C$ and $C^{\ast}$ have Lebesgue
measure zero, and $\overline{B^{\ast}}$ is compact (since so is $G$)
\cite{Kobayashi,Klingenberg}. Consider the compact space 
$$F=\{ (X,Y,Z)\in\overline{B^{\ast}}^3:\ \exp(X)\,\exp(Y)=
\exp(Z)\}\subset\g^3\;,$$
and for each $X\in\overline{B^{\ast}}$ the compact slice
$$F_X=\{ (Y,Z)\in\g^2:\ (X,Y,Z)\in F\}\subset\g^2\;.$$
Smoothness on $F$ and $F_X$ will refer to the smoothness obtained by considering these
spaces as subspaces of $\g^3$ and $\g^2$, respectively.

Let $\iota:\g^2\rightarrow\g^2$ be the involution $(Y,Z)
\mapsto (Z,Y)$. For $a=\exp(X)$ we also have the smooth map ${\mathcal J}_X:
B\cap L_a^{-1}B\ar F_X$ given by ${\mathcal J}_X(g)= (\log(g),\log(ag))$, where 
$\log=\exp^{-1}:B\ar B^{\ast}$. Let $W_X={\mathcal J}_X(B\cap L_a^{-1}B) \subset
F_X$.

\begin{lem} [{\cite[Proposition 2.2]{Alv5}}]\label{lemma:JX}We have:
\begin{enumerate}
\item[(i)] $W_X$ is open in $F_X$ and ${\mathcal J}_X: B\cap L_a^{-1}B \ar W_X$ 
is a diffeomorphism.
\item[(ii)] $\iota(F_X)=F_{-X}$, and the diagram
$$\begin{CD}
B\cap L_a^{-1}B           @>{\mathcal J}_X>>           F_X\\
@V{L_a}VV                                          @VV{\iota}V\\
B\cap L_aB                @>{\mathcal J}_{-X}>>        F_{-X}
\end{CD}$$
is commutative.
\end{enumerate}
\end{lem}    

For $X,Y\in\overline{B^{\ast}}$ let $W_{X,Y}= {\mathcal J}_X(B
\cap L_a^{-1}B\cap L_b^{-1}B)\subset F_X$, where $a=\exp(X)$ and $b=\exp(Y)$. We
have the diffeomorphism ${\mathcal J}_{X,Y}={\mathcal J}_Y{\mathcal J}_X^{-1}: W_{X,Y}
\ar W_{Y,X}$.

Let $\Delta$ be the unique
biinvariant volume form on $G$ such that $\int_G\Delta=1$, which defines a Haar
measure $\mu$ on $G$. 
Then for each $X\in\overline{B^{\ast}}$ let $\mu_X$ be the Borel measure on $F_X$,
concentrated on $W_X$, where it corresponds to $\mu$ by ${\mathcal J}_X$.

\begin{cor}[{\cite[Proposition 2.3]{Alv5}}]\label{cor:mu}
We have:
\begin{enumerate}
\item[(i)] $\mu_X(F_X)=\mu_X(W_X)=\mu_X(W_{X,Y})=\mu(B\cap L_a^{-1}B\cap
L_b^{-1}B)\\ =\mu(B\cap L_a^{-1}B)=\mu(G)=1$
\item[(ii)] $\mu_X$ corresponds to $\mu_{-X}$ by $\iota:F_X\ar F_{-X}$.
\item[(iii)] $\mu_X$ corresponds to $\mu_Y$ by ${\mathcal J}_{X,Y}:W_{X,Y}\ar 
W_{Y,X}$.
\end{enumerate}
\end{cor}

Let $I=[0,1]$, and define continuous maps $\sigma,\eta:F\times I\ar G$ by
setting
$$\sigma(\xi,t)=\exp(tZ)\;,$$
$$\eta(\xi,t)=
\begin{cases}
\exp(2tX) & \text{if $t\in I_1=[0,1/2]$} \\
\exp(X)\,\exp((2t-1)Y) & \text{if $t\in I_2=[1/2,1]$}\;,
\end{cases}$$
where $\xi=(X,Y,Z)\in F$. The map $\sigma$ is smooth, and so are the restrictions of
$\eta$ to each $F\times I_i$ ($i=1,2$).

\begin{lem} [{\cite[page 178]{Alv5}}]\label{lemma:calHj} 
There is a finite open cover $Q_1,\ldots,Q_k$ of $F$, and continuous
maps $\calH_j:Q_j\times I\times I\ar G$ with smooth restrictions to each $Q_j\times I_i\times
I$, $i=1,2$, $j=1,\ldots,k$, so that 
$$\calH_j(\cdot,\cdot,0)=\sigma|_{Q_j\times I}\;,\quad
\calH_j(\cdot,\cdot,1)=\eta|_{Q_j\times I}\;,$$ 
$$\calH_j(\xi,0,s)=e\quad\text{for
all}\quad s\in I\quad\text{and}\quad\xi\in Q_j\;,$$
$$\calH_j(\xi,1,s)=\exp(Z)\quad\text{for all}\quad s\in I\quad\text{and}\quad\xi=
(X,Y,Z)\in Q_j\;.$$  \end{lem} 

\begin{lem} \label{lemma:Hj}
For each $j=1,\ldots,k$ there exists a unique continuous map 
$$H_j:N\times Q_j\times I\times I\ar N$$ 
with smooth restrictions to each $N\times Q_j\times I_i\times I$, $i=1,2$, 
such that 
\begin{enumerate}
\item[(i)] $D_NH_j(z,\xi,t,s)=D_N(z)\,\calH_j(\xi,t,s)$, 
\item[(ii)] $H_j(z,\xi,0,s)=z$,
\item[(iii)] $(d/dt)H_j(z,\xi,t,s)\in\nu$ for $t\neq 1/2$.
\end{enumerate}
Moreover for $\xi=(X,Y,Z)\in Q_j$ we have
\begin{enumerate}
\item[(iv)] $H_j(\cdot,\xi,1,0)=\phi_Z$,
\item[(v)] $H_j(\cdot,\xi,1,1)=\phi_Y\phi_X$
\item[(vi)] $H_j(z,\xi,1,s)\in D_N^{-1}(D(z)\,\exp(Z))$ for all $z\in N$ and
all $s\in I$. \end{enumerate}
\end{lem}

\begin{proof} It is completely similar to the proofs of Lemmas~3.1 and~3.2 in \cite{Alv5}.
\end{proof}

Therefore, for all $\xi=(X,Y,Z)\in Q_j$, $H_j(\cdot,\xi,1,\cdot):N\times I
\ar N$ is an $\widetilde{\calF}$-integrable homotopy of $\phi_Z$ to
$\phi_Y\phi_X$ \cite{Kacimi83}. Hence the corresponding homotopy operator in $\Omega$
preserves the filtration, and thus its $(0,-1)$-bihomogeneous component $k_{j,\xi}: \Omega
\ar \Omega$ satisfies
$$(\phi_X^{\ast}\phi_Y^{\ast}-\phi_Z^{\ast})_{0,0}=d_{0,1}k_{j,\xi}+k_{j,\xi} 
d_{0,1}\;.$$ 
Define the operators $\rho,\lambda:\Omega\ar \Omega$ by setting
$$\rho(\alpha)=\int_{\overline{B^{\ast}}}\phi_X^{\ast}\alpha\,\Delta^{\ast}(X)\;,
\quad\lambda(\alpha)=\int_{\overline{B^{\ast}}}\Phi_X\alpha\,\Delta^{\ast}(X)\;,$$
where $\Delta^{\ast}=\exp^{\ast}\Delta$ and $\Phi_X$ is the homogeneous operator of
degree $-1$ on $\Omega$ associated to the homotopy $\phi_{tX}$ ($t\in I$) \cite{BottTu}. The
operators $\rho$ and $\lambda$ are linear homogeneous of degrees $0$ and $-1$,
respectively, satisfying $\rho-\text{id} =d\lambda+\lambda d$. Moreover, since
$\phi_{tX}$ preserves the pair of foliations $\left(\calG ,\widetilde{\calF}\right)$ 
(because $X^{\nu}$ is an infinitesimal transformation of $\left(\calG ,\widetilde{\calF}\right)$),
$\Phi_X$ reduces the filtration at most by a unit. Therefore the bihomogeneous operators
$\rho_1\equiv\rho_{0,0\ast}$ and
$\lambda_1\equiv\lambda_{-1,0\ast}$ on $E_1$ satisfy
$\rho_1-\text{id} =d_1\lambda_1+\lambda_1d_1$.

For $\alpha\in \Omega$ and $X\in\overline{B^{\ast}}$, by Lemma \ref{lemma:JX} and
Corollary \ref{cor:mu} we have
\begin{eqnarray*}
\phi_X^{\ast}\rho(\alpha)&=&\int_{F_X}\phi_X^{\ast}\phi_Y^{\ast}\alpha\,d\mu_X(Y,Z)\;, \\ 
\rho(\alpha)&=&\int_{W_{X,-X}}\phi_Y^{\ast}\alpha\, d\mu_X(Y,Z) \\
&=&\int_{W_{-X,X}}\phi_Y^{\ast}\alpha\, d\mu_{-X}(Y,Z) \\
&=&\int_{F_{-X}}\phi_Y^{\ast}\alpha\, d\mu_{-X}(Y,Z) \\
&=&\int_{F_X}\phi_Z^{\ast}\alpha\, d\mu_X(Y,Z)\;.
\end{eqnarray*} 
So
\begin{equation}\label{e:phiXrho-rho}
(\phi_X^{\ast}\rho-\rho)\alpha=\int_{F_X}(\phi_X^{\ast}\phi_Y^{\ast}-
\phi_Z^{\ast})\alpha\, d\mu_X(Y,Z)\;.
\end{equation}
Take a smooth partition of unity $f_1,\ldots,f_k$ of $F$ subordinated to the open
cover $Q_1,\ldots,Q_k$. Then the $f_j(X,\cdot,\cdot)$
form a partition of unity of $F_X$ subordinated to the open cover given by the
slices 
$$Q_{j,X}=\{(Y,Z)\in \g^2:\ (X,Y,Z)\in Q_j\}\,.$$ 
Let $\Psi_X:\Omega\ar \Omega$
be the $(0,-1)$-bihomogeneous linear operator given by
$$\Psi_X\alpha=\sum_{j=1}^k\int_{Q_{j,X}}k_{j,\xi}\alpha\,f_j(\xi)\,d\mu_X(Y,Z)\;,$$ where
$\xi=(X,Y,Z)$ for each $(Y,Z)\in Q_{j,X}$. From \eqref{e:phiXrho-rho} we get 
\begin{equation}\label{e:PsiX}
(\phi_X^{\ast}\rho-\rho)_{0,0}=d_{0,1}\Psi_X+\Psi_Xd_{0,1}\;.
\end{equation}

\begin{lem} \label{lemma:PsiX}
$\Psi_X\alpha$ depends continuously on $X\in\overline{B^{\ast}}$ for each
$\alpha\in \Omega$ fixed.
\end{lem} 

\begin{proof} It is completely analogous to the proof of Lemma 3.3 in \cite{Alv5}.
\end{proof}

\begin{lem} \label{lemma:phitX}
For $\alpha\in \Omega$, $X\in\g $ and $t\in\R $ we have
$$\phi_{tX}^{\ast}\alpha=\alpha+\int_0^t\phi_{sX}^{\ast}\theta_X\alpha\,ds =
\alpha+\theta_X\int_0^t\phi_{sX}^{\ast}\alpha\,ds\;.$$
\end{lem} 

\begin{proof} It is completely analogous to the proof of Lemma 3.4 in \cite{Alv5}.
\end{proof}

\begin{lem} \label{lemma:E2iso}
$\rho_1(E_1)=(E_1)_{\theta_1=0}$, and
$$\begin{CD}\rho_{1\ast}:E_2 @>\cong>> H((E_1)_{\theta_1=0})\;.\end{CD}$$
\end{lem} 

\begin{proof} First, we shall prove that $\rho_1(E_1)\subset(E_1)_{\theta_1=0}$.
Take any $\alpha\in\ker(d_{0,1})$ defining $[\alpha]\in E_1$. If $[\alpha]\in
\rho_1(E_1)$, we can suppose 
$\alpha=\rho_{0,0}\beta$ for some $\beta\in\ker(d_{0,1})$. Then 
\begin{equation}\label{e:phiXPsiX}
(\phi_X^{\ast})_{0,0}\alpha-\alpha=d_{0,1}\Psi_X\beta\quad\text{for all}\quad
X\in\overline{B^{\ast}} \end{equation}  by \eqref{e:PsiX}. Thus Lemmas~\ref{lemma:PsiX}
and~\ref{lemma:phitX} yield  
$$(\theta_X)_{0,0}\alpha=d_{0,1} \left(\Psi_X\beta-(\theta_X)_{0,0}\int_0^1
\Psi_{sX}\beta\,ds \right)$$ as in \cite[page 181]{Alv5}. Therefore
$\rho_1([\alpha])\in(E_1)_{\theta_1=0}$.

Let $\iota:(E_1)_{\theta_1=0}\ar E_1$ be the inclusion map. If $[\alpha]
\in (E_1)_{\theta_1=0}$, since $(\theta_X)_{0,0}$ depends linearly on $X\in
\g $, there is a linear map $X\mapsto\beta_X$ of $\g$ to $\Omega$ so
that $(\theta_X)_{0,0}\alpha = d_{0,1}\beta_X$ for all $X\in\g$. Thus by
Lemma \ref{lemma:phitX} we get
$$\rho_{0,0}\alpha=\alpha+d_{0,1}\int_{\overline{B^{\ast}}}\int_0^1(\phi_{sX}^{\ast})_{0,0}
\beta_X\,ds\,\Delta^{\ast}(X)\;,$$
yielding $\rho_1\iota = \text{id} $. In particular $\rho_1(E_1)=
(E_1)_{\theta_1=0}$. We also have $\iota\rho_1-\text{id} =d_1\lambda_1+\lambda_1
d_1$, and the result follows. \end{proof}

\begin{proof}[End of the proof of Theorem~\ref{thm:E2}] Since $G$ is compact,
the representation $\theta_{\g}$ is semisimple \cite[Sections~4.4 and~5.12]{GHV:III}.
So $$H((E_1)_{\theta_1=0})\cong H(\g )\otimes(E_1^{0,\cdot})_{\theta_1=0}$$
by \cite[Theorem V in Section~4.11, and Section~5.26]{GHV:III}. The
result now follows from Lemma \ref{lemma:E2iso}. \end{proof}

\subsection{Relation between $H^\cdot(\calF)$ and $E_2$}

\begin{thm}\label{thm:HF=E2}
With the above notations, $H^{\cdot}(\calF )\cong E_2^{0,\cdot}$.
\end{thm}

To begin with the proof of Theorem~\ref{thm:HF=E2}, the section
$s:M\ar N$ defines a homomorphism $(s^{\ast})_1:E_1^{0,\cdot}\ar H^{\cdot}(\calF )$ since $s^{\ast}
d_{0,1}=d_\calF s^{\ast}$. By restricting $(s^{\ast})_1$, we get
$(s^{\ast})_2: E_2^{0,\cdot}=(E_1^{0,\cdot})_{\theta_1=0}\ar H^{\cdot}(\calF )$. We
will prove that $(s^{\ast})_2$ is an isomorphism.

For any $X\in\g$ set $s_X=\phi_Xs:M\ar N$, which is an
embedding, but not a section of $\pi_N$ in general. Nevertheless $s_X(M)=s(M)
\,\exp(X)$. Analogously to $s$, the map $s_X$
also defines $(s_X^{\ast})_1:E_1^{0,\cdot}\ar H^{\cdot}(\calF )$. Let
$U_X$ be the neighborhood of $s_X(M)$ given by
$$U_X = \bigcup_{Y\in B^{\ast}} \phi_Ys_X(M) = s_X(M)\, B =
D_N^{-1}(\exp(X)\, B)\;.$$
For each $X\in\g$ and each $x\in M$, $s_X$ defines an isomorphism
$$\begin{CD}s_{X\ast}: T_x\calF @>\cong>> T_{s_X(x)}\widetilde{\calF}\;.\end{CD}$$
So
$$\begin{CD}
s_X^{\ast}:\bigwedge T_{s_X(x)}\widetilde\calF^{\ast}@>\cong>> \bigwedge T_x\calF^{\ast}\;.
\end{CD}$$
For each $\omega\in \Omega(\calF )$, let $\omega_X$ be the unique smooth
section of $\bigwedge T\widetilde\calF^{\ast}$ over $s_X(M)$ such that
$s_X^{\ast}\omega_X=\omega$. Define $T_X\omega\in 
\Omega^{0,\cdot}(\calG |_{U_X})$ by
setting
$$(T_X\omega)(\phi_Ys_X(x))=(\phi_{-Y}^{\ast})_{0,0}\omega_X(s_X(x))$$
for $Y\in B^{\ast}$ and $x\in M$. This is well defined since $(x,Y)\mapsto
\phi_Ys_X(x)$ is a diffeomorphism of $M\times B^{\ast}$ onto $U_X$. Moreover
$d_{0,1}T_X=T_Xd_{\calF}$ since $d_{0,1}\equiv d_{\widetilde{\calF}}$ on
$\Omega^{0,\cdot}\equiv\Omega^\cdot\left(\widetilde{\calF}\right)$, and
$(\phi_Ys_X)^{\ast}\widetilde{\calF}=\calF$ for all
$X,Y\in\g$. Therefore $T_X$ defines a map $T_{X\ast}: H^{\cdot}(\calF )\ar
E_1^{0,\cdot}(\calG |_{U_X})$.

The inclusion map $\iota_X:U_X\ar N$ induces
$(\iota_X^{\ast})_1:E_1^{0,\cdot}\ar E_1^{0,\cdot}(\calG |_{U_X})$.

\begin{lem} \label{lemma:sXiotaX}
If $\zeta\in(E_1^{0,\cdot})_{\theta_1=0}$, then $T_{X\ast}(s_X^{\ast})_1\zeta=
(\iota_X^{\ast})_1\zeta$.
\end{lem} 

\begin{proof} By Lemma \ref{lemma:E2iso} we have $\rho_1\zeta=\zeta$. We thus
can choose forms $\alpha,\gamma\in \Omega^{0,\cdot}$ such that $d_{0,1}\alpha=0$,
$\zeta=[\alpha]$, and $\alpha=\rho_{0,0}\alpha+d_{0,1}\gamma$. Then
\eqref{e:phiXPsiX} yields
$$
(\phi_Y^{\ast})_{0,0}(\alpha-d_{0,1}\gamma)-(\alpha-d_{0,1}\gamma)=d_{0,1}\Psi_Y\alpha
$$ 
for any $Y\in B^{\ast}$. So
\begin{equation}\label{e:phiY*alpha-alpha}
(\phi_Y^{\ast})_{0,0}\alpha-\alpha=d_{0,1}(\Psi_Y\alpha+
(\phi_Y^{\ast})_{0,0}\gamma-\gamma)\;. 
\end{equation}
Clearly $(s_X^{\ast}\alpha)_X=\alpha|_{s_X(M)}$. Hence 
\begin{eqnarray*}
(T_Xs_X^{\ast}\alpha)(\phi_Ys_X(x))&=&(\phi_{-Y}^{\ast})_{0,0}(\alpha(s_X(x))) \\
&=&(\alpha+d_{0,1}(\Psi_{-Y}\alpha+(\phi_{-Y}^{\ast})_{0,0}\gamma-\gamma))
(\phi_Ys_X(x))
\end{eqnarray*} 
by \eqref{e:phiY*alpha-alpha}. But since each $\phi_Ys_X(M)$ is $\widetilde{\calF}$-saturated,
$d_{0,1}\equiv d_{\widetilde{\calF}}$ commutes with the restriction to each $\phi_Ys_X(M)$. Therefore we
get $$T_Xs_X^{\ast}\alpha=\alpha+d_{0,1}\eta_X$$
on $U_X$, where $\eta_X$ is the $(0,\cdot)$-form on $U_X$ defined by
$$\eta_X(\phi_Ys_X(x))=(\Psi_{-Y}\alpha+(\phi_{-Y}^{\ast})_{0,0}\gamma-\gamma)
(\phi_Ys_X(x))\;,$$ which finishes the proof. \end{proof}

Since $G$ is compact, there is a finite sequence $0=X_1,X_2,\ldots,X_l$ of
elements of $\g $ such that
$$G=B\cup\exp(X_2)\, B\cup\cdots\cup\exp(X_l)\, B\;.$$
Let $U_j=U_{X_j}$ $T_j=T_{X_j}$, $s_j=s_{X_j}$ and $\iota_j=\iota_{X_j}$ for $j=1,\ldots,l$. Then
$N=U_1\cup\cdots\cup U_l$. Let $h_1,\ldots,h_l$ be a smooth partition of unity of
$G$ subordinated to the open cover $\exp(X_1)\,B,\ldots,\exp(X_l)\,B$ so that 
$h_1(e)=1$. Then $D_N^\ast h_1,\ldots,D_N^\ast h_l$ is a partition of unity of N
subordinated to  $U_1,\ldots,U_l$.

For $\omega\in \Omega(\calF)$, define $T\omega\in \Omega^{0,\cdot}$ by setting
$$T\omega=\sum_{j=1}^l D_N^\ast h_j\, T_j\omega\;.$$
Since each $D_N^\ast h_j$ is constant along the leaves of $\widetilde{\calF}$, we get
$d_{0,1}T=Td_\calF .$ So $T$ defines a map 
$T_{\ast}:H^{\cdot}(\calF )\ar E_1^{0,\cdot}$.

\begin{lem} \label{lemma:Tsz=z}
If $\zeta\in(E_1^{0,\cdot})_{\theta_1=0}$, then $T_{\ast}(s^{\ast})_1\zeta=\zeta$.
\end{lem} 

\begin{proof} For each $X\in\g$, let $(\phi_X^{\ast})_1:E_1\ar E_1$
be the homomorphism defined by $\phi_X^{\ast}$ ($(\phi_X^{\ast})_1\equiv
(\phi_X^{\ast})_{0,0\ast}$). Since $s_X=\phi_Xs$, by \eqref{e:phiXPsiX}
we have
$$(s_X^{\ast})_1\zeta=s_1^{\ast}(\phi_X^{\ast})_1\zeta=s_1^{\ast}\zeta\;.$$
Therefore, by Lemma \ref{lemma:sXiotaX}, 
$$(\iota_j^{\ast})_1\zeta=T_{j\ast}(s_j^{\ast})_1\zeta=T_{j\ast}(s^{\ast})_1\zeta$$
for $j=1,\ldots,l$. 
So, if $\zeta=[\alpha]$ for $\alpha\in \Omega^{0,\cdot}$ with $d_{0,1}\alpha=0$, there
is some $\beta_j\in \Omega^{0,\cdot}$ for each $j$ such that
$\alpha-T_js^{\ast}\alpha=d_{0,1}\beta_j$ over $U_j$. 
Let
$$\beta=\sum_{j=1}^lD_N^\ast h_j\, \beta_j\in \Omega^{0,\cdot}\;.$$
Since each $D_N^\ast h_j$ is constant on the leaves of $\widetilde{\calF}$ and $d_{0,1}
\equiv d_{\widetilde{\calF}}$, we get
\begin{eqnarray*}
d_{0,1}\beta&=&\sum_{j=1}^l D_N^\ast h_j\, d_{0,1}\beta_j \\
&=&\sum_{j=1}^l D_N^\ast h_j\, (\alpha-T_js^{\ast}\alpha) \\
&=&\alpha-Ts^{\ast}\alpha\;,
\end{eqnarray*}
and the proof is complete.
\end{proof} 

\begin{lem} \label{lemma:s*2surjective}
$(s^{\ast})_2:E_2^{0,\cdot}\ar H^{\cdot}(\calF )$ is surjective.
\end{lem} 

\begin{proof} Take any $\omega\in \Omega(\calF )$ with $d_\calF \omega=0$, and take any 
function $f\geq 0$ compactly supported in $B$ such that $\int_Bf(g)\,\Delta(g)=1$.
Then $\alpha=D_N^\ast f\, T_1\omega$ is a $(0,\cdot)$-form compactly supported in
$U_1$ and satisfying $d_{0,1}\alpha =0$. So $\alpha$ defines a class $\zeta\in
E_1^{0,\cdot}$. We shall prove that $(s^{\ast})_1\rho_1\zeta=[\omega]$.

For $x\in M$ and $Y\in B^{\ast}$ we have
$$\alpha(\phi_Ys(x))=f(\exp(Y))\,(\phi_{-Y}^{\ast})_{0,0}(\omega_{X_1}(s(x)))\;.$$
So
$$((\phi_Y^{\ast})_{0,0}\alpha)(s(x))=f(\exp(Y))\,\omega_{X_1}(s(x))\;,$$
yielding
\begin{eqnarray*}
(\rho_{0,0}\alpha)(s(x))&=&\int_{B^{\ast}}((\phi_Y^{\ast})_{0,0}\alpha)(s(x))\,
\Delta^{\ast}(Y) \\
&=&\omega_{X_1}(s(x))\int_{B^{\ast}}f(\exp(Y))\,\Delta^{\ast}(Y) \\
&=&\omega_{X_1}(s(x))\int_Gf(g)\,\Delta(g) \\
&=&\omega_{X_1}(s(x))\;.
\end{eqnarray*} 
Therefore $s^{\ast}\rho_{0,0}\alpha=s^{\ast}\omega_{X_1}=\omega$, and the proof follows.
\end{proof}

\begin{cor}\label{cor:T*HFsubsetE2}
$T_{\ast}(H^{\cdot}(\calF ))\subset E_2^{0,\cdot}$.
\end{cor}

\begin{proof} It follows directly from Lemmas~\ref{lemma:Tsz=z} and~\ref{lemma:s*2surjective}.
\end{proof}

\begin{proof}[End of the proof of Theorem~\ref{thm:HF=E2}] By Corollary~\ref{cor:T*HFsubsetE2} we can
consider $T_{\ast}:H^{\cdot}(\calF)\rightarrow E_2^{0,\cdot}$. By Lemma \ref{lemma:Tsz=z} we have
$T_{\ast}(s^{\ast})_2=\text{id} $. On the other hand, $(s^{\ast})_2T_{\ast}=\text{id}$
because $(D_N^\ast h_1)(s(x))=1$  for all $x\in M$ since $h_1(e)=1$. So $(s^{\ast})_2$ is an
isomorphism.  \end{proof}

\begin{cor}\label{cor:H1F=H1G}
$H^1(\calF )\cong H^1(\calG )$ and $\calH^1(\calF )\cong
\calH^1(\calG )$.
\end{cor}

\begin{proof} Theorem~\ref{thm:E2} yields $E_2^{2,0}\cong H^2(\g )\otimes
E_2^{0,0}=0$ since $\g $ is compact semisimple. So $E_2^{0,1}=
E_{\infty}^{0,1}\cong H^1(\calG )$ canonically. Then 
$H^1(\calF )\cong H^1(\calG )$ as topological vector spaces by 
Theorem~\ref{thm:HF=E2}, 
obtaining also $\calH^1(\calF )\cong \calH^1(\calG )$. \end{proof}

\begin{cor}\label{cor:H2F}
$H^2(\calF )$ and $\calH^2(\calF )$ are of finite dimension if and only if so are $H^2(\calG )$
and  $\calH^2(\calG )$, respectively.
\end{cor}

\begin{proof} The leaves of $\calG $ are dense since so are the leaves of
\calF. Thus $H^0(\calG )\cong\Bbb R$, yielding
$E_2^{\cdot,0}\cong H^{\cdot}(\g )$ by Theorem~\ref{thm:E2}. On the other hand,
$H^1(\g )=H^2(\g )=0$ since $\g$ is compact semisimple \cite{Poor}. So
$E_i^{1,\cdot}=E_i^{2,\cdot}=0$ for $2\leq i\leq\infty$ by Theorem~\ref{thm:E2}. Hence
$E_3^{0,2}=E_2^{0,2}\cong H^2(\calF )$ (using Theorem~\ref{thm:HF=E2}), and
$E_3^{3,0}=E_2^{3,0}\cong H^3(\g )$. Therefore, since
$$E_{\infty}^{0,2}=E_4^{0,2}=\ker(d_3:E_3^{0,2}\ar E_3^{3,0})\;,$$
$H^2(\calG )\cong E_{\infty}^{0,2}$ can be identified to the kernel of some
continuous homomorphism of $H^2(\calF )$ to $H^3(\g )$, and the result
follows. \end{proof}

\begin{cor}\label{cor:css,k=1,2}
Suppose $M$ is oriented. Let $\iota_i:K_i\ar M$, $i=1,2$,  
be smooth immersions of closed
oriented manifolds of complementary dimension which define homology classes of $M$ with
non-trivial  intersection. Let $\G_i$ be the image of the composite
$$\begin{CD}\pi_1(K_i) @>\pi_1(\iota_i)>> \pi_1(M) @>h>> G\;.\end{CD}$$ 
Suppose the group generated by $\G_1\cup\G_2$  
is not dense in $G$. If $1\leq k=\dim K_2\leq 2$, then $\dim\calH^k(\calF)=\infty$. \end{cor}

\begin{proof} The result follows directly applying Corollaries~\ref{cor:H1F=H1G},~\ref{cor:H2F}, 
and~\ref{cor:suspension2} to $\calG $.  \end{proof}

\begin{cor}\label{cor:css}
Suppose $M$ is oriented. Let $\iota_i:K_i\ar M$, $i=1,2$,  
be smooth immersions of closed
oriented manifolds of complementary dimension which define homology classes of $M$ with 
non-trivial 
intersection. Let $\G_i$ be the image of the composite
$$\begin{CD}\pi_1(K_i) @>\pi_1(\iota_i)>> \pi_1(M) @>h>> G\;.\end{CD}$$ 
Suppose the group generated by $\G_1\cup\G_2$  
is not dense in $G$. If $\iota_1$ is transverse to \calF,  then 
$\dim\calH^k(\calF )=\infty$ for $k=\dim K_2$.
\end{cor}

\begin{proof} By Corollary~\ref{cor:css,k=1,2}, we can assume $k>2$. Let
$F^lH^{\cdot}(\calG )$ and 
$F^l\calH^{\cdot}(\calG )$, $l=0,1,2,\ldots$,  
be the filtrations of $H^{\cdot}(\calG )$ and $\calH^{\cdot}(\calG )$
induced by \eqref{e:filtration}. We have
$$H^\cdot(\calG )/F^1H^\cdot(\calG )\cong E_{\infty}^{0,\cdot}\subset E_2^{0,\cdot}\cong
H^\cdot(\calF)\;,$$ where both isomorphisms preserve the topologies, and $E_\infty^{0,\cdot}$
is a closed subspace of $E_2^{0,\cdot}$. (The last isomorphism follows from
Theorem~\ref{thm:HF=E2}.) So $\calH^\cdot(\calG )/F^1\calH^\cdot(\calG )$ can be injected into
$\calH^\cdot(\calF )$, and it is  enough to prove that $\calH^k(\calG )/F^1\calH^k(\calG )$ is of
infinite dimension.

 This is a special case of the setting of Theorem~\ref{thm:subfoliations} and 
Corollaries~\ref{cor:suspension1} and~\ref{cor:suspension2}. The proofs of those results yield 
linearly independent classes $\zeta_m\in\calH^k(\calG )$. In this case, we shall prove
that the $\zeta_m$ are also linearly independent modulo $F^1\calH^k(\calG )$.

Consider the pull-back bundles $\iota_i^\ast N$ over $K_i$. The canonical maps  
$\tilde{\iota}_i:\iota_i^\ast N\ar N$ are immersions 
transverse to $\calG $, which thus can be considered as transversely regular immersions of
$(\iota_i^\ast N,\calG _i)$ into $(N,\calG )$, where 
$\calG _i=\tilde{\iota}_i^\ast\calG $. We can assume the $\iota_i$ intersect each other
transversely, and thus the $\tilde{\iota}_i$ intersect transversely in $\calG $.

Let $\widetilde{H}^\cdot\subset H^{\cdot}(\calG )$ and 
$\widetilde{\calH}^{\cdot}\subset{\calH}^{\cdot}(\calG )$ be the subspaces given by the classes
that have representatives supported in $\pi_N^{-1}(U)$ for any open subset $U\subset M$
containing $\iota_1(K_1)$. Set $F^1\widetilde{H}^\cdot= \widetilde{H}^\cdot\cap F^1H^\cdot(\calG )$
and  $F^1\widetilde{\calH}^\cdot=
\widetilde{\calH}^{\cdot}\cap F^1{\calH}^{\cdot}(\calG )$. Since
$\zeta_m\in\widetilde{\calH}^k$ by Remark~\ref{rem:subfoliations}, it is enough to prove that
the $\zeta_m$ are linearly independent modulo $F^1\widetilde{\calH}^k$.
Hence, according to the proof of Theorem~\ref{thm:subfoliations}, it is enough to prove that
$\iota_2$ can be chosen so that $\tilde{\iota}_2^\ast\left(F^1\widetilde{\calH}^k\right)=0$ where
$\tilde{\iota}_2^\ast: {\calH}^\cdot(\calG )\rightarrow\calH^\cdot(\calG _2)$. In fact
we shall prove the stronger property that the choice of $\iota_2$ can be made so that 
$\tilde{\iota}_2^\ast\left(F^1\widetilde{H}^k\right)=0$ for $\tilde{\iota}_2^\ast:H^\cdot(\calG )\ar
H^\cdot(\calG _2)$. 
 
Since
$\iota_1$ is transverse to \calF,  we can choose $\iota_2$ such that, for some open subset
$U\subset M$  containing $\iota_1(K_1)$, each connected component of
$\iota_2(K_2)\cap U$ is contained in some leaf of \calF. So, for every leaf $L_2$ of $\calG _2$,
the connected components of  $\tilde{\iota}_2(L_2)\cap\pi_N^{-1}(U)$ are contained in leaves of
$\widetilde{\calF}$, yielding $\tilde{\iota}_2^\ast\alpha=0$ over $\tilde{\iota}_2^{-1}\pi_N^{-1}(U)$ 
for any $\alpha\in F^1\Omega^\cdot(\calG )$. Moreover $U$ and $\iota_2$ can be
chosen so that the connected components of $\iota_2^{-1}(U)$ are contractible; thus
$\tilde{\iota}_2^{-1}\pi_N^{-1}(U)\equiv\iota_2^{-1}(U)\times G$ canonically, where the
slices $\iota_2^{-1}(U)\times \{\ast\}$ are the leaves of
the restriction $\calG _{2,U}$ of $\calG _2$ to $\tilde{\iota}_2^{-1}\pi_N^{-1}(U)$. Hence
$H^l(\calG _{2,U})=0$ for $l>0$. Finally, the above choices can be made so that, for some open
subset $V\subset M$, we have $\iota_1(K_1)\cap V=\emptyset$, $U\cup V=M$, and each connected
component of $\iota_2^{-1}(U\cap V)$ is contractible. Thus, as above, $H^l(\calG _{2,U\cap
V})=0$ for $l>0$, where $\calG _{2,U\cap V}$ is the restriction of  $\calG _2$ to
$\tilde{\iota}_2^{-1}\pi_N^{-1}(U\cap V)$. 
Therefore, by using
the Mayer-Vietoris type spectral sequence (cf. \cite{Kacimi83})
$$\cdots\ar H^{l-1}(\calG _{2,U\cap V})\ar H^l(\calG _2)
\ar H^l(\calG _{2,U})\oplus H^l(\calG _{2,V})
\ar H^l(\calG _{2,U\cap V})\rightarrow\cdots$$
and since $k>2$, we get 
\begin{equation}\label{e:HG2V}H^k(\calG _2)\cong H^k(\calG _{2,V})\end{equation} 
by the restriction homomorphism.
 
Now any $\xi\in F^1\widetilde{H}^k$ can be defined by a leafwise closed form 
$\alpha\in\Omega^k(\calG )$ supported in $M\setminus V$ with $\alpha+d_{\calG }\beta\in 
F^1\Omega^k(\calG )$ for some $\beta\in\Omega^{k-1}(\calG )$. Then 
$\tilde{\iota}_2^\ast(\alpha+d_{\calG }\beta)$ is supported in
$\tilde{\iota}_2^{-1}\pi_N^{-1}(V)$, where it is the $\calG _2$-leafwise derivative of
$\tilde{\iota}_2^\ast\beta$. So   $\tilde{\iota}_2^\ast\xi$ is mapped to zero in
$H^k(\calG _{2,V})$, and thus $\tilde{\iota}_2^\ast\xi=0$ by \eqref{e:HG2V}, which finishes the
proof. \end{proof}

\section{Case of foliations on nilmanifolds $\G\backslash H$ defined by 
normal subgroups of $H$ \label{sec:nilpotent}}

The goal of this section is to prove Theorem~\ref{thm:nilpotent}. It will be done by induction, which
needs leafwise reduced cohomology with coefficients in a vector bundle with a flat
\calF-partial connection. Thus we shall prove a more general theorem by taking arbitrary coefficients.

 For a foliation \calF\  on a manifold $M$ and a vector bundle $V$ over $M$, a flat 
\calF\ -partial connection on $V$ can be defined as a flat connection on the
restriction of $V$ to the leaves  whose local coefficients are smooth on each
foliation chart of \calF\  on $M$. So the corresponding  de~Rham derivative $d_\calF$ with
coefficients in $V$ preserves smoothness on $M$; i.e. $d_\calF$ preserves $\Omega(\calF ,V)=\cinf
(\bigwedge T\calF^\ast\otimes V)$. Then   $\calH^\cdot(\calF ,V)$ can be defined in the same way as
$\calH^\cdot(\calF )$ by using  $(\Omega(\calF ,V),d_\calF )$ instead of $(\Omega(\calF ),d_\calF )$.

Consider the following particular case. Let $H$ be a simply connected nilpotent Lie 
group, $K\subset H$ a normal connected subgroup, and $\G\subset H$ a discrete
uniform subgroup whose projection to $H/K$ is dense. Then let \calF\  be the foliation
on the nilmanifold $M=\G\backslash H$ defined as the quotient of the foliation
$\widetilde{\calF}$  on $H$ whose leaves are the translates of $K$. In this case, $M$ is closed and the
leaves of \calF\  are dense. Let $\widetilde{V}$ be an $H\times K$-vector bundle over $H$ for the left
action of $H\times K$ on $H$ given by $(h,k)h'=hh'k^{-1}$, $(h,k)\in H\times K$ and $h'\in H$. We
also consider the induced left actions of $H$ and $K$ on $H$. The space of  $H$-invariant
sections of $\widetilde{V}$ will be denoted by $\cinf\left(\widetilde{V}\right)_H$, and the subspaces of
invariant sections will be denoted in a similar way for other actions. Suppose  $\widetilde{V}$ is endowed
with an $H\times K$-invariant flat $\widetilde{\calF}$-partial connection,  and let $V$
be the induced vector bundle on $M$ with the induced flat \calF\ -partial
connection.  The structure of $H\times K$-vector bundle on $V$ canonically defines an 
action of $\frak{k}$ on $\cinf\left(\widetilde{V}\right)_H$, where $\frak{k}$ is the Lie algebra
of $K$. Moreover the induced differential map on 
$\bigwedge\frak{k}^\ast\otimes \cinf\left(\widetilde{V}\right)_H$ corresponds to
$d_{\widetilde{\calF}}$  by the canonical injection of this space in $\Omega(\calF ,V)$.

\begin{thm}\label{thm:nilpotent coefficients}
With the above notations, 
$\calH^\cdot(\calF ,V)\cong H^\cdot\left(\frak{k},\cinf\left(\widetilde{V}\right)_H\right)$.
 \end{thm}

The result will follow by induction on the
codimension $q$ of \calF.

For $q=0$ and $V$ the trivial line bundle, this is just a well known theorem of K.~Nomizu 
\cite{Nomizu}. If $q=0$ and $V$ is arbitrary, the result still follows with the obvious adaptation of the
arguments in \cite{Nomizu}.

Suppose $q>0$ and the result is true for foliations of codimension less than $q$. The proof has
two cases.

\begin{case}\label{case:1} Assume $K\cap\G=1$. The group $\G$ is
nilpotent since so is $H$, thus the center of $\G$ is non-trivial. Let $a$
be a non-trivial element in the center of $\G$. By the universal property of
Mal'cev's completion \cite{Mal'cev}, there exists a one dimensional connected
subgroup $L$ of the center of $H$ containing $\langle a\rangle$ as a discrete
uniform subgroup. $L$ is isomorphic to $\Bbb R$ since $H$ is simply connected.
Let $H_1=H/L$, and $\G_1=\G/\langle a\rangle$. Clearly $\G_1$ is
canonically injected in $H_1$ as a discrete uniform subgroup. We get $L\cap K=1$
because $\langle a\rangle\cap K=1$, and thus there is a canonical injection of
$K$ into $H_1$ as a normal subgroup, defining a foliation $\calF_1$ on the
nilmanifold $M_1=\G_1\backslash H_1$. $\calF_1$ is a foliation of the type
considered in the statement of this theorem, of codimension $q-1$, but observe
that the canonical injection of $K$ into $H_1$ may not have trivial intersection
with $\G_1$. The projection $H/\langle a\rangle \ar H_1$ is canonically an
$S^1$-principal bundle (considering $S^1\equiv  L/\langle
a\rangle$), so the induced map $\pi : M \ar M_1$ is also an $S^1$-principal bundle in a canonical
way. Then $V$ canonically is an $S^1$-vector bundle so that  the partial connection is
invariant, and thus induces the vector bundle $V_1=V/S^1$ over  $M_1$  with the corresponding
flat $\calF_1$-partial connection. The lifting of $V_1$ to $H_1$ is 
$\widetilde{V}_1=\widetilde{V}/L$, which satisfies the same properties as $\widetilde{V}$
with respect to $K_1$  instead of $K$.

For each $x\in M_1$ and each $m\in\Bbb Z$, define
\begin{eqnarray*}
C_{m,x}&=&\{f\in \cinf (\pi^{-1}(x),\C ):\  
f(y\theta)=f(y)\,\mathrm{e}^{2\pi m\theta\mathrm{i}}\ \\
&&\text{for all}\ y\in \pi^{-1}(x)\ \text{and all}\ \theta\in S^1\equiv\R/\Z \}\;.
\end{eqnarray*} 
It is easy to see that $$C_m=\bigsqcup_{x\in M_1}C_{m,x}$$ is a
one-dimensional $\Bbb C$-vector bundle over $M_1$ in a canonical way. For $m\in\Bbb Z$,
define also 
\begin{eqnarray*}
\Omega(\calF ,V\otimes\C )^m&=&\{\alpha\in \Omega(\calF ,V\otimes\C ):\ \alpha(y\theta)=\alpha(y)\,
\mathrm{e}^{2\pi m\theta\mathrm{i}}\ \\  
&&\text{for all}\ y\in\pi^{-1}(x)\ \text{and all}\ \theta\in
S^1\}\;, \end{eqnarray*} 
and similarly define $\cinf\left(\left(\widetilde{V}/\langle a\rangle\right)\otimes\C\right)^m$
considering the $S^1$-principal bundle $H/\langle a\rangle\ar H_1$. By the
Fourier series expression for functions on $S^1$, we get that $\Omega(\calF,V\otimes\C )$
is the $\cinf $ closure of  $$\bigoplus_{m\in\Bbb Z}\Omega(\calF,V\otimes\C )^m\;.$$ It can be
easily seen that there is a canonical isomorphism \begin{equation}\label{e:isomorphism}
\Omega(\calF_1,V_1\otimes C_m)\cong \Omega(\calF ,V\otimes\C )^m
\end{equation}
defined by $\pi^{\ast}$ and the canonical identity $$\cinf (C_m)\equiv
\cinf (M,\C )^m\;.$$ 
Since \calF\  is preserved by 
the $S^1$-action on $M$, $d_\calF $
preserves each $\Omega(\calF ,V\otimes\C)^m$ and corresponds to $d_{\calF_1}$ by
\eqref{e:isomorphism}. By induction $$\calH^{\cdot}(\calF_1,V_1\otimes
C_m)\cong H^{\cdot}\left(\frak{k},\cinf\left(\widetilde{V}_1\otimes\tilde{C}_m\right)_{H_1}\right)\;.$$
But 
$$
\cinf\left(\widetilde{V}_1\otimes\widetilde{C}_m\right)_{H_1}\cong \cinf\left(\left(\widetilde{V}/\langle
a\rangle\right)\otimes\C\right)^m_{H/\langle a\rangle}
$$ 
canonically, which is obviously 
trivial if $m\neq 0$. But $C_0$ is the trivial complex line bundle, so
\begin{eqnarray*}
\calH^{\cdot}(\calF ,V\otimes\C ) & \cong & \calH^{\cdot}(\calF_1,V_1\otimes C_0) \\
& \cong & H^{\cdot}\left(\frak{k},\cinf\left(\widetilde{V}_1\otimes\C\right)_{H_1}\right) \\
& \cong & H^{\cdot}\left(\frak{k},\cinf\left(\widetilde{V}\otimes\C\right)_H\right)\;.
\end{eqnarray*} 
\end{case}

\begin{case} In the general case, let $G=H/K$ and $\G_1$ the projection of
$\G$ to $G$. We use Mal'cev's construction for the pair $(G,\G_1)$.
It yields a simply connected nilpotent Lie group $H_1$ containing $\G_1$
as a discrete uniform subgroup, and a surjective homomorphism
$D_1:H_1\ar G$ which is the identity on
$\G_1$. The kernel $K_1$ of $D_1$ defines a foliation $\calG $ of codimension $q$ on
the nilmanifold $M_1=\G_1\backslash H_1$, and we have $K_1\cap\G_1=1$. So $\calG $ is the
type of foliation we have considered in Case~\ref{case:1}.

$\calG $ is the classifying foliation for foliations with transverse structure given
by $(G,\G_1)$. So there is a smooth map $f:M\ar M_1$ which is transverse
to $\calG $ and so that $\calF =f^{\ast}\calG $. In this particular case, $f$ can
be constructed in the following way. By the universal property of Mal'cev's
construction, the surjective homomorphism of $\G$ to $\G_1$ can be uniquely
extended to a surjective homomorphism $\tilde{f}:H\ar H_1$, which defines a
map  $f : M \ar M_1$. We have $D_1\tilde{f}=D$. So $K$ is projected onto $K_1$, and thus $\calF
=f^{\ast}\calF_1$. Moreover $f$ is a locally trivial bundle with fiber the nilmanifold
$P/(P\cap\G)$, where $P$ is the kernel of $\tilde{f}$.

Fix a vector subbundle $\nu\subset T\calF $
which is complementary to the subbundle $\tau\subset T\calF $ of vectors that are tangent to
the fibers of $f$. Then we get a canonical isomorphism
$$\bigwedge T\calF^{\ast}\otimes V\cong\bigwedge\nu^{\ast}\otimes\bigwedge\tau^{\ast}\otimes
V\;,$$ yielding a bigrading of $\Omega(\calF ,V)$ given by
$$
\Omega^{u,v}(\calF ,V) = \cinf\left(\bigwedge^u\nu^{\ast}\otimes\bigwedge^v\tau^{\ast}\otimes
V\right)\;.
$$ 
Consider the filtration of $\Omega(\calF,V)$ given by the differential subspaces
$$F^k\Omega(\calF,V)=\bigoplus_{u\geq k}\Omega^{u,\cdot}(\calF,V)\;,$$
which depend only on \calF\  and $V$; in fact they could be defined without using $\nu$. This 
filtration induces a spectral sequence $(E_i,d_i)$ converging to $H^{\cdot}(\calF ,V)$, whose 
terms $(E_0,d_0)$ and $(E_1,d_1)$ can be described as follows. 
The derivative $d_\calF $ decomposes as sum of bihomogeneous operators $d_{\calF ,0,1}$,  
$d_{\calF ,1,0}$ and $d_{\calF ,2,-1}$, where each double subindex indicates the corresponding
bidegree. These operators satisfy identities which are similar to those in
\eqref{e:di,j 1} and~\eqref{e:di,j 2}, yielding 
$$(E_0,d_0) \equiv (\Omega(\calF ,V),d_{\calF ,0,1})\;,$$
$$(E_1,d_1) \equiv (H(\Omega(\calF ,V),d_{\calF ,0,1}),d_{\calF ,1,0\ast})\;.$$

Let $\frak{k}_1$ be the Lie algebra of $K_1$. Each $X\in\frak{k}_1$ canonically defines
a vector field $X_1$ on $M_1$ which is tangent to the leaves of $\calF_1$. Let $X_{\nu}$ be the
unique vector field on $M$ which is a section of $\nu$ and projects to $X_1$. For $\alpha\in
\Omega^{0,v}(\calF )$ and $s\in \cinf (V)$, define $\theta_X(\alpha\otimes s)$ to be the
$(0,\cdot)$-component of $$\theta_{X_{\nu}}\alpha\otimes s+\alpha\otimes\nabla_{X_{\nu}}s\;,$$
where $\nabla$ denotes the flat \calF\ -partial connection of $V$. It can be easily checked that
$\theta_Xd_{\calF ,0,1}=d_{\calF ,0,1}\theta_X$. 
So $\theta_X$ defines an operator, also denoted by $\theta_X$, on $E_1^{0,\cdot}$. In
this way, we get a representation $\theta$ of $\frak{k}_1$ on $E_1^{0,\cdot}$, and a
canonical isomorphism $E_2^{u,v} \cong H^u(\frak{k}_1,\theta)$.

Define 
$$V_{1,y}=H^\cdot\left(f^{-1}(y),V|_{f^{-1}(y)}\right)\;,\quad y\in M_1\;,$$
$$V_1=\bigsqcup_{y\in M_1}V_{1,y}\;,$$ and let $\widetilde{V}_1$ be the lifting of $V_1$ to
$H_1$. It is easy to see that $\widetilde{V}_1$ canonically is a $H_1\times K_1$-vector
bundle over the $H_1\times K_1$-manifold $H_1$ with an $H_1\times K_1$-invariant flat 
$\widetilde{\calF}_1$-partial
connection. (The fibers of $\widetilde{V}_1$ are of finite
dimension since the fibers of $f$ are compact.) It is also easily seen that there is
a canonical isomorphism 
$\cinf (V_1) \cong E_1^{0,\cdot}$. 
Moreover the
representation of $\frak{k}_1$ on $E_1^{0,\cdot}$ corresponds to the representation
of $\frak{k}_1$ on $\cinf\left(\widetilde{V}_1\right)$ defined by the flat partial connection
of  $\widetilde{V}_1$. So $$E_2^{u,\cdot} \cong
H^u\left(\frak{k}_1,\cinf\left(\widetilde{V}_1\right)\right)\cong H^u(\calF_1,V_1)\;.$$

Let ${\mathcal E}_i$ be the quotient of $E_i$ over the closure $\overline{0_i}$ of its trivial subspace.
Then 
$$
{\mathcal E}_2^{u,\cdot}\cong \calH^u(\calF_1,V_1)\cong
H^u\left(\frak{k}_1,\cinf\left(\widetilde{V}_1\right)_{H_1}\right)
$$ 
by Case~\ref{case:1}.

If the above filtration is restricted to the space of differential forms in $\Omega(\calF,V)$
whose lifting to $H$ is $H$-left invariant, we get a spectral sequence
$(\overline{E}_i,\overline{d}_i)$ converging to
$H^\cdot\left(\frak{k}_1,\cinf\left(\widetilde{V}\right)_H\right)$, and there is a canonical homomorphism 
$(\overline{E}_i,\overline{d}_i) \ar (E_i,d_i)$ of spectral sequences. Analogously, we have a canonical
isomorphism $$\overline{E}_2^{u,\cdot} \cong H^u(\frak{k}_1,\cinf (V_1)_{H_1})\;.$$ So the composite
$\overline{E}_2 \ar E_2 \ar {\mathcal E}_2$
is an isomorphism, and thus
$E_2 \cong \overline{E}_2 \oplus \overline{0_2}$
as differential complexes. Then
$E_3\cong \overline{E}_3 \oplus H(\overline{0_2},d_2)$, 
yielding $H(\overline{0_2},d_2) \cong \overline{0_3}$, and the above decomposition is of
differential complexes. We get
$E_4\cong \overline{E}_4 \oplus H(\overline{0_3},d_3)$. 
Continuing with these arguments, we finally obtain  
$E_i \cong \overline{E}_i \oplus \overline{0_i}$
as topological differential complexes for $i\geq 2$, and thus
$$H^\cdot(\calF ,V) \cong E_\infty \cong \overline{E}_\infty \oplus \overline{0_\infty}\;.$$
Hence 
$$\calH^\cdot(\calF ,V) \cong \overline{E}_\infty \cong H^\cdot(\frak{k},\cinf (V)_H)$$
as desired.
\end{case} 

\begin{rem}
For general Lie foliations with dense leaves and nilpotent structural Lie algebra, the
classifying foliations are of the type considered in Theorem~\ref{thm:nilpotent}. On the one
hand, if the ambient manifold is closed and the classifying map can be chosen to be a fiber
bundle, then a spectral sequence argument shows that the leafwise reduced cohomology is of finite
dimension. On the other hand, if the classifying map has unavoidable singularities, then they
should correspond to handles on the leaves and the leafwise reduced cohomology is of infinite
dimension by Corollary~\ref{cor:homology2}. \end{rem}

\end{document}